\documentclass[10pt]{article}

\usepackage{amssymb,amsthm,bbm}
\usepackage[centertags]{amsmath}

\usepackage{mathabx}
\let\abxast=\ast
\usepackage{MnSymbol}

\let\ast=\abxast
\usepackage[mathscr]{euscript}
\usepackage{tikz-cd}
\usepackage{hyperref}

\usepackage{rotating}

\newtheorem{thm}{Theorem}[subsection]
\newtheorem{lem}[thm]{Lemma}

\newtheorem{prop}[thm]{Proposition}
\newtheorem{cor}[thm]{Corollary}
\theoremstyle{definition}
\newtheorem{defn}[thm]{Definition}
\newtheorem{exmp}[thm]{Example}
\newtheorem{rem}[thm]{Remark}

\renewcommand{\inf}{$\infty$-}

\newcommand{\gap}[2]{(#1, #2)}                      % The cartesian gap map
\newcommand{\cogap}[2]{\lfloor #1, #2 \rfloor}      % The cocartesian gap map

\newcommand{\pp}{\,\square\,}      % The pushout product (mat's shortcut)
\newcommand{\ppz}{\,\square_Z\,}   % The pushout product over Z
\newcommand{\pprodovr}[3]{#2 \, \square_{#1} \, #3}        % The relative pushout product

\newcommand{\ph}[2]{\langle #1, #2 \rangle}                % The pullback hom (mat's shortcut)
\newcommand{\intph}[2]{\llangle #1, #2 \rrangle}           % The internal pullback hom (mat's shortcut)
\newcommand{\ini}{0}				% command for initial object	(mat)
\newcommand{\term}{1}				% command for terminal object		(mat)

\newcommand{\intperp}{{\begin{sideways}$\!\Vdash$\end{sideways}}}		% Internal orthogonality (mat)
\newcommand{\iperp}{\ \intperp\ }							% Internal orthogonality with space on sides (mat)
\newcommand{\map}[2]{[#1, #2]}

\newcommand{\Cart}{\mathrm{Cart}}
\newcommand{\intmap}[2]{\lsem #1, #2 \rsem}

\newcommand{\pbmark}{\ar[dr, phantom, "\ulcorner" very near start, shift right=1ex]}
\newcommand{\pomark}{\ar[ul, phantom, "\lrcorner" very near start, shift right=1ex]}
\newcommand{\op}{^\mathrm{op}}
\newcommand{\join}{\star}

\DeclareMathOperator*{\colim}{colim}

\DeclareMathOperator{\id}{id}

\newcommand{\cA}{\mathscr{A}}

\newcommand{\cC}{\mathscr{C}}
\newcommand{\cD}{\mathscr{D}}
\newcommand{\cE}{\mathscr{E}}

\newcommand{\cL}{\mathscr{L}}
\newcommand{\cM}{\mathscr{M}}
\newcommand{\cN}{\mathscr{N}}

\newcommand{\cP}{\mathscr{P}}

\newcommand{\cR}{\mathscr{R}}
\newcommand{\cS}{\mathscr{S}}

\makeatletter
\pgfqkeys{/tikz/commutative diagrams}{
  row sep/.code={\tikzcd@sep{row}{#1}{}},
  column sep/.code={\tikzcd@sep{column}{#1}{}},
  bo row sep/.code={\tikzcd@sep{row}{#1}{between origins}},
  bo column sep/.code={\tikzcd@sep{column}{#1}{between origins}},
  bo column sep/.default=normal,
  bo row sep/.default=normal,
}
\def\tikzcd@sep#1#2#3{% re-defintion of original package macro!
  \pgfkeysifdefined{/tikz/commutative diagrams/#1 sep/#2}%
    {\pgfkeysalso{/tikz/#1 sep={\ifx\\#3\\1*\else1.7*\fi\pgfkeysvalueof{/tikz/commutative diagrams/#1 sep/#2},#3}}}%
    {\pgfkeysalso{/tikz/#1 sep={#2,#3}}}}

\title{A Generalized Blakers-Massey Theorem}
\author{Mathieu Anel\footnote{Department of Philosophy, Carnegie Mellon University, mathieu.anel@gmail.com} , \ \ 
Georg Biedermann\footnote{Universidad del Norte, Barranquilla, Colombia, gbm@posteo.de} , \ \ 
Eric Finster\footnote{University of Birmingham, Department of Computer Science, ericfinster@gmail.com} , \\
and Andr\'{e} Joyal\footnote{CIRGET, UQ\`AM. joyal.andre@uqam.ca} 
}

\begin{document}

\maketitle

\begin{abstract}
  We prove a generalization of the classical connectivity theorem of
  Blakers-Massey, valid in an arbitrary higher topos and with respect
  to an arbitrary \emph{modality}, that is, a factorization system
  $(\cL, \cR)$ in which the left class is stable by base change. We
  explain how to rederive the classical result, as well as the recent
  generalization of \cite{ChachSchererWerndli}.  Our proof is inspired
  by the one given in Homotopy Type Theory in \cite{FinsterLumsdaine}.
\end{abstract}

\setcounter{tocdepth}{2}
\tableofcontents

\section{Introduction}
\label{sec:introduction}

The classical Blakers-Massey theorem, sometimes known as the homotopy
excision theorem, is one of the most fundamental facts in homotopy
theory.  Given a homotopy pushout diagram of spaces
\[
  \begin{tikzcd}
    A \ar[r, "g"] \ar[d, "f"'] & C \ar[d] \\
    B \ar[r] & D \pomark
  \end{tikzcd}
\]
such that the map $f$ is $m$-connected and the map $g$ is
$n$-connected, the theorem tells us that the canonical map
$A\to B\times_{D} C$ to the homotopy pullback is in fact
$(m+n)$-connected. (We direct readers, who find themselves surprised
by the statement, to Remark~\ref{rem:conn-conv} for an 
explanation of our indexing conventions for connected maps.) 
Among other things, the theorem gives rise to the
Freudenthal suspension theorem and, thus, to stable homotopy theory.

Recently, a new proof of this theorem was found in the context of
\emph{homotopy type theory}, a formal system originating in
constructive mathematics and computer science which has been shown to
provide an elementary axiomatization of homotopy theoretic reasoning
\cite{hottbook}.  One pleasant feature of this proof is that it is
entirely homotopy invariant, neither relying on a particular model
of homotopy types such as topological spaces or simplicial sets, nor
requiring more sophisticated mathematical machinery such as
transversality arguments or homology calculations. A second and
perhaps more surprising consequence is that, written as it is in a
formal language, it becomes subject to automatic verification by a
computer. The interested reader may consult \cite{FinsterLumsdaine},
where just such formalization is described in detail.

The reasoning formalized by homotopy type theory is generally thought
to serve as an ``internal language'' for a particular class of higher
categories, namely the $\infty$-topoi as developed by Rezk~\cite{RezkTopos} and Lurie~\cite{LurieHT}.  
This is to say that each operation of the
logic has a corresponding interpretation as a higher categorical
construction.  As a consequence, the original proof of
\cite{FinsterLumsdaine} may be translated into the language of higher
category theory, an undertaking which is carried out in unpublished
work of Rezk \cite{Rezk:blakers-massey}, and which we revisit in this
article. Our result is a much generalized theorem, applying not only
to spaces, but to an arbitrary $\infty$-topos.  As we
will see in the companion article \cite{GBM2}, the generalized theorem
can be applied to an appropriate presheaf topos and yields an analogue
of the Blakers-Massey theorem in the context of Goodwillie's calculus of functors. 

In order to pursue these sorts of applications, however, we will need
to further generalize the theorem, beyond simply placing it in an
abstract context.  While on the face of it, the theorem speaks about
the connectivity of certain maps, it turns out that we may in fact
replace the property of ``connectedness'' with any other property of
morphisms which behaves sufficiently like it.  The central observation
here is that the $n$-connected maps form the left class of a
\emph{factorization system} $(\cL, \cR)$ on the category of spaces
with the additional property that the left class $\cL$ is stable under
base change.  We refer to a factorization system satisfying this
condition as a \emph{modality}, a term originating in the
literature on type theory \cite[Section 7.7]{hottbook}.

Concretely, then, our main theorem is the following:
\vspace*{2mm}

\noindent
{\bf Theorem~\ref{thm:gen-bm}}
Let $\cE$ be an \inf topos and $(\cL,\cR)$ a modality on $\cE$.
Write $\Delta h : A \to A \times_B A$ for the diagonal of a
map $h : A \to B \in \cE$ and $- \ppz -$ for the
pushout product in the slice category $\cE_{/Z}$.
Given a pushout square
\[
  \begin{tikzcd}
    Z \ar[d, "f"'] \ar[r, "g"] & Y \ar[d] \\
    X \ar[r] & W \pomark
  \end{tikzcd}
\]
in $\cE$, suppose that $\Delta f \ppz \Delta g \in\cL$. Then the
canonical map $\gap{f}{g}: Z\to X\times_W Y$ is also in $\cL$.
\vspace*{2mm}

In fact, a similar generalization of the Blakers-Massey theorem was
recently obtained by Chacholski, Scherer, and Werndli in
\cite{ChachSchererWerndli}, and their work provided inspiration for
the statement of our main result.  Our techniques, however, are
quite different: their method involves the manipulation of
\emph{weak cellular inequalities} of spaces, as introduced in
\cite{farjoun1995cellular}, whereas we focus on $\infty$-topos theoretic tools
such as descent. 
(It is possible though to interpret our results as proving weak cellular inequalities for morphisms.)
The present work can be seen as a synthesis and generalization of two new approaches to a classical result: via weak cellular inequalities and via higher topos theory/homotopy type theory. Overall, the necessary input from classical homotopy theory has become almost invisible.

We also would like to draw the attention of the reader to the results in Subsection~\ref{subsec:decentLcart} where descent properties associated to the left class of a modality are proved. These results are important for our proof of the generalized Blakers-Massey theorem but might be of independent interest.

Let us now turn to an outline of the paper.  Section 2 fixes our
higher categorical conventions and recalls some elementary facts which
will be used throughout the paper.  We briefly review the definition
of an $\infty$-topos, including the axiom of descent.  Section 3 begins by
introducing the notion of a factorization system, as well as the
pushout product and pullback hom, two constructions which prove
convenient for manipulating orthogonality relations between maps in a
category.  In a cartesian closed category such as an $\infty$-topos,
orthogonality can be strengthened to an internal version, and we
explore some of the properties of factorization systems compatible
with this internalization, of which modalities will prove to be
examples.  We then give a short treatment of the
$n$-connected/$n$-truncated factorization system in an $\infty$-topos.  This
archetypal example of a modality will be important for extracting the
classical Blakers-Massey theorem from our generalized version.  Next, we introduce the notion of a modality itself, providing a number of examples and
deriving some elementary properties, including the Dual Blakers-Massey
theorem.  We conclude the section by deriving what turns out to be the
most crucial property of modalities for our purposes: the descent theorem
for $\cL$-cartesian squares.  Section 4 then turns to the proof of the
generalized Blakers-Massey theorem itself, finishing with the derivation of the
classical theorem, as well as that of Chacholski-Scherer-Werndli.
\vspace*{2mm}

\noindent {\bf Acknowledgments:} 
The authors would like to thank J\'er\^ome Scherer for helpful discussions about his joint paper \cite{ChachSchererWerndli}, Karol Szumi\l o, Sarah Yeakel and the referee for helpful comments which have prompted us to clarify and simplify our arguments.

The first author has received funding from the European Research Council under the European Community's Seventh Framework Programme (FP7/2007-2013 Grant Agreement n$^{\circ}$263523) and the support of the Air Force Office of Scientific Research through MURI grant FA9550-15-1-0053.
The second author and this project have received funding from the European Union’s Horizon 2020 research and innovation programme under Marie Sk\l odowska-Curie grant agreement No 661067.  
The second author also acknowledges support from the project ANR-16-CE40-0003 ChroK.
The third author has been supported by the CoqHoTT ERC Grant 64399.
The fourth author has been supported by the NSERCC grant 371436.

\section{Higher Topoi}
\label{sec:higher-topoi}

\subsection{Higher Categorical Conventions}
\label{sec:higher-cats}

Throughout this paper, we employ the language of higher category
theory, considering only homotopy invariant constructions.  Moreover,
we use terminology which reflects this convention: by \emph{category}
we will always mean an $(\infty,1)$-category, saying $1$-category
explicitly to refer to an ordinary category if the occasion so arises.
In particular, we will from now on refer to an $\infty$-topos simply as a topos.
Similarly, we say simply \emph{limit} and \emph{colimit} for the
higher categorical version, what would ordinarily be called the
homotopy limit or homotopy colimit.  All mapping spaces are ``derived'',
and composition of morphisms is associative up to coherent higher
homotopy. 

For readers unfamiliar with the literature on higher category theory,
we have tried hard to make the paper nonetheless accessible.  Indeed,
our arguments involve only the elementary manipulation of
homotopy limits and colimits, they are, in a sense, model independent.
A reader more familiar with the theory of model categories should have
no trouble interpreting our results in, for example, a simplicial
model category.  Of course, for a more precise discussion of the
relationship between the higher categorical approach and the model
category theoretic one, we refer the reader to \cite{LurieHT}.

We will use the word {\em space} to refer to an abstract homotopy
type, what is often called an $\infty$-groupoid in the higher
categorical literature.  The reader is free to keep in mind any
preferred model for these objects, such as topological spaces (compactly generated Hausdorff) or
simplicial sets up to weak homotopy equivalence, but none of our arguments will
depend on such a choice.  We write $\cS$ for the category of spaces.

For two objects $X$ and $Y$ in a category $\cC$, we write $\map{X}{Y}$
for the space of maps between $X$ and $Y$.  The words map, morphism
and arrow will be used interchangeably, as is common.  For a category
$\cC$, we let $\cC^\to$ denote its category of arrows.  By  an {\em
  isomorphism} in $\cC$ will refer to a morphism which is invertible
in $\cC$ in the higher categorical sense: for example, in $\cS$
the isomorphisms correspond to the weak homotopy equivalences when
homotopy types are modeled as topological spaces.
We will often write``$X=Y$'' to mean
that two objects $X$ and $Y$ of $\cC$ are isomorphic,
when the isomorphism is clear. 
Similarly, we will write $f = g$ to mean 
that two maps $f,g:X\to Y$ are homotopic, when
the homotopy is clear. 
As eg. in the statement of Proposition~\ref{appendixpp-adj2},
we may also write
$f = g$ to mean that two maps are naturally isomorphic in the arrow category 
$\cC^\to$ when the isomorphisms are clear. As the former "$=$" is a special case of the latter we hope this does not cause confusion.

We encourage the interested reader to consult \cite{hottbook} for the homotopy type theory
perspective on the equality relation.

We will write $\ini$ for initial and $\term$ for terminal objects.

Given a finite family of maps $f_i:X\to Y_i$ where $1 \leq i \leq n$
in a category $\cC$, we will write
\[
(f_1,\dots, f_n) : X \to Y_1\times \dots \times Y_n.
\]
for the canonical map from $X$ to the product of the $Y_i$.
Dually, for a  finite family $f^i:X_i\to Y$ where 
$1 \leq i \leq m$, we write
\[
  \lfloor f^1,\dots, f^m \rfloor : X_1 \sqcup \dots \sqcup X_m \to Y.
\]
for the canonical map from the coproduct of the $X_i$ to $Y$.  More
generally, for any doubly indexed family $f_j^i:X_i \to Y_j$ where
$1 \leq i \leq m$ and $1 \leq j \leq n$, we have an induced ``total" map
$$
T(f_j^i) : X_1 \sqcup \dots \sqcup X_m \to Y_1\times \dots \times Y_n.
$$
and we leave it to the reader to check that this map obeys the commutation relation
\begin{align*}
T(f_j^i) & =\left\lfloor (f_1^1,\dots, f_n^1),\dots , (f_1^m,\dots, f_n^m)\right\rfloor    \\
& = \begin{bmatrix}
    f^1_1       & f^1_2 &  \dots & f^1_n \\
    f^2_1       & f^2_2 &  \dots & f^2_n \\
    \hdotsfor{4} \\
    f^m_1       & f^m_2 & \dots & f^m_n
\end{bmatrix}\\
&= \left( \lfloor f^1_1,\dots, f^m_1 \rfloor, \dots , \lfloor f^1_n,\dots, f^m_n \rfloor\right).
\end{align*}

The following special cases of the above 
notation will occur frequently
enough that they will merit some special terminology.  Suppose we are 
given a commutative square
\[
\begin{tikzcd}
  Z \ar[d, "f"'] \ar[r, "g"] & Y \ar[d, "k"] \\
  X \ar[r, "h"'] & W 
\end{tikzcd}
\]
in a category $\cC$. Taking the pushout of the diagram
$X \leftarrow Z \to Y$ or the pullback of the diagram
$X \to W \leftarrow Y$, we obtain two canonical maps
which, using the previous notation, will be denoted by
\begin{align*}
  \gap{f}{g} : Z \to X \times_W Y \\
  \cogap{h}{k} : X \sqcup_Z Y \to W
\end{align*}
We will refer to the first of these two maps as the
\emph{cartesian gap map}, or merely the \emph{gap map}.  The
second will be referred to as the \emph{cocartesian gap map}, or
more briefly, the \emph{cogap map}.  This notation is in fact
mildly abusive since the maps in question depend on the data
of the entire commutative square.  That is to say, the first
map is a special case of our general notation when regarded
in the slice category $\cC_{/W}$ and the second
in the coslice category $\cC_{Z/}$.  In practice, however, the
remaining maps will be clear from the context.

\subsection{Topoi and Descent}
\label{sec:descent-topoi}

There are many equivalent characterizations of the notion of a
\emph{topos}, but for the purposes of this article we will adopt the
position that a topos is simply a category satisfying a certain
collection of \emph{exactness conditions}, that is, compatibilities
between limits and colimits.  While this is perhaps not the most
profound point of view on the subject, it nonetheless has the
benefit of practicality, making explicit the constructions which can be
performed, and hence will be adequate for our purposes here.

We will need a couple of elementary facts about
\emph{presentable} categories, for whose complete theory we refer the
reader to \cite[Ch. 5]{LurieHT}.  A presentable
category has all limits and colimits, and a functor $F : \cC \to \cD$
between presentable categories preserves all colimits
if and only if it has a right adjoint.  We say that the colimits in
a presentable category $\cC$ are
{\it universal} if the base change functor 
  $$f^* : \cC_{/Y} \to \cC_{/X}$$ 
  preserves colimits for any map $f:X\to Y$ in $ \cC$.  In this case,
  the functor $f^*$ admits a right adjoint $f_*$ by the previous remarks.
In particular, the base change functor $A\times (-): \cC\to  \cC_{/A}$
has a right adjoint 
  $$\Pi_A:\cC_{/A}\to \cC$$ 
for every object $A\in \cC$. It follows that the category $ \cC$ is cartesian closed with internal hom 
  $$\intmap{A}{B}=\Pi_A(A\times B,p_A)$$
for every $A,B\in  \cC$, where $p_A$ is the projection onto $A$.

We will say that a morphism $\alpha:f\to g$
in the arrow category $\cC^\to$ is {\it cartesian} if the corresponding
square in $\cC$ is cartesian. The composite of two cartesian morphisms
is cartesian, since the composite of two cartesian squares is cartesian.
We will denote by $\Cart(\cC^\to)$ the (non-full)
subcategory of cartesian morphisms of $\cC^\to$.

\begin{defn} \label{generaldescent}
We say that a cocomplete category $\cC$ satisfies the {\em descent principle} if the 
subcategory $\Cart(\cC^\to)$ is closed under colimits.
\end{defn}

The closure condition in the definition means two things:
colimits exist in the subcategory $\Cart(\cC^\to)$
and they are preserved by the inclusion functor $\Cart(\cC^\to)\to \cC^\to$.
More precisely, a diagram $D:I\to \cC^\to$ is the same thing as a natural transformation
$\alpha:D_0\to D_1$ between two diagrams $D_0,D_1:I\to \cC$.
The diagram $D$ belongs to the subcategory $\Cart(\cC^\to)$ 
if and only if the natural transformation $\alpha:D_0\to D_1$ is \emph{cartesian}:
that is, if the naturality square

 \[
    \begin{tikzcd}
      D_0(i) \ar[r, "D_0(f)"] \ar[d, "\alpha(i)"'] \pbmark & D_0(j) \ar[d, "\alpha(j)"] \\
      D_1(i) \ar[r, "D_1(f)"'] & D_1(j) 
    \end{tikzcd}
  \]
is cartesian for every arrow $f:i\to j$ in the category $I$.
The colimit of $D:I\to \cC^\to$ is the map $\colim(\alpha):\colim D_0 \to \colim D_1$.
The descent principle implies that the square 
\[
    \begin{tikzcd}
      D_0(i) \ar[r, "\iota_0(i)"] \ar[d, "\alpha(i)"'] \pbmark & \colim D_0 \ar[d, "\colim(\alpha)"] \\
      D_1(i) \ar[r, "\iota_1(i)"'] & \colim D_1 
    \end{tikzcd}
  \]
is cartesian for every $i\in I$, where $\iota_0(i)$ and $\iota_1(i)$ are the canonical maps.
The principle also implies that a square
\[
    \begin{tikzcd}
     \colim D_0  \ar[d, "\colim(\alpha)"'] \ar[r, "u_0"]   & A \ar[d, "f"] \\
     \colim D_1  \ar[r, "u_1"'] & B 
    \end{tikzcd}
  \]
is cartesian if and only if the square 
\[
    \begin{tikzcd}
      D_0(i) \ar[rr, "u_0\iota_0(i)"] \ar[d, "\alpha(i)"']  && A \ar[d, "f"] \\
      D_1(i) \ar[rr, "u_1\iota_1(i)"'] && B 
    \end{tikzcd}
  \]
is cartesian for every $i\in I$.

The applications of the descent principle in the present work
are all consequences of the following 

\begin{lem} \label{lemmadescent} 
If a category $\cC$ satifies the descent principle then for every pushout
  \[
    \begin{tikzcd}
      f \ar[r, "\alpha"] \ar[d, "\beta"'] & g \ar[d, "\gamma"] \\
      h \ar[r, "\delta"'] & k \pomark
    \end{tikzcd}
  \]
in $\cC^\to$, such that $\alpha$ and $\beta$ are cartesian, $\gamma$ and $\delta$ are also cartesian.
\end{lem}

It may be worth spelling out what the lemma says in the category $\cC$
itself.  A square of arrows in $\cC^\to$ corresponds 
to a cubical diagram in $\cC$ as follows:
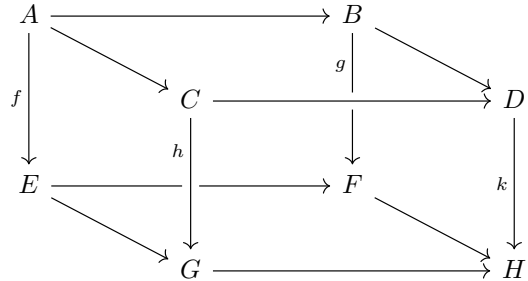
\[
  \begin{tikzcd}[bo column sep=large]
    A \ar[rr] \ar[dd, "f"'] \ar[dr] & & B \ar[dd, "g"', near start] \ar[dr] & \\
    & C \ar[rr, crossing over] & & D \ar[dd, "k"'] \\
    E \ar[rr] \ar[dr] & & F \ar[dr]  & \\
    & G \ar[rr, crossing over] \ar[from=uu, crossing over, "h"', near start] & & H
  \end{tikzcd}
\]
The hypothesis of Lemma \ref{lemmadescent} then requires that the
top and bottom horizontal squares are pushouts and that the back and
left squares ($\alpha$ and $\beta$ in the definition above) are
pullbacks.  The conclusion then asserts that the
front and right squares ($\gamma$ and $\delta$) are pullbacks as well.
For the category of spaces this fact is well known and often referred to as Mather's cube lemma~\cite{mather1976pull}.

\begin{defn}
  We say that a category $\cE$ is a \emph{topos} if
  \begin{enumerate}
  \item $\cE$ is presentable,
  \item colimits in $\cE$ are universal, and
  \item $\cE$ satisfies the descent principle.
  \end{enumerate}
\end{defn}

\begin{exmp}
  The category of spaces $\cS$ is a topos, as is the category of
  presheaves $[\cC^{\op}, \cS]$  for any small category
  $\cC$. More generally, any left-exact localization of a presheaf
  category is a topos, and this in fact completely characterizes the
  class of topoi.  See \cite[prop. 6.1.3.10]{LurieHT}.
\end{exmp}

Let us give a simple application of descent. 

\begin{defn}\label{def:mono}
We will say that a map $f : A \to B$ in a topos $\cE$ is a \emph{monomorphism}
if the square
\[
  \begin{tikzcd}
    A \ar[d, equal] \ar[r, equal] & A \ar[d, "f"] \\
    A \ar[r, "f"'] & B
  \end{tikzcd}
\]
is cartesian.  
\end{defn}

This concept will be more thoroughly treated in Section~\ref{sec:conn-trun}.  For example, in the category of spaces $\cS$, a map $f : A \to B$ is a monomorphism if and only if it is (weakly equivalent to) an inclusion of a union of path components of $B$ into $B$.

\begin{prop}
  \label{prop:mono-pushout}
  Consider a pushout square 
  \[
    \begin{tikzcd}
      A \ar[d, "f"', rightarrowtail] \ar[r, "h"]  & C \ar[d, "g"] \\
      B \ar[r, "k"'] & D \pomark
    \end{tikzcd}
  \]
in a topos in which the arrow $f$ is a monomorphism.  Then $g$ is a monomorphism and the square is cartesian.
\end{prop}

\begin{proof}
  Consider the cube:
  \[
    \begin{tikzcd}
      A \ar[rr, "h"] \ar[dd, equal] \ar[dr, equal] & & C \ar[dd, equal] \ar[dr, equal] & \\
      & A \ar[rr, "h", near start, crossing over] & & C \ar[dd, "g"] \\
      A \ar[rr, "h"', near end] \ar[dr, "f"', rightarrowtail] & & C \ar[dr, "g"'] & \\
      & B \ar[rr, "k"'] \ar[from=uu, crossing over, "f"', near start, rightarrowtail] & & D
    \end{tikzcd}
  \]
  The top square is trivially cocartesian and the back square is
  trivially cartesian.  Note also that the left side is cartesian,
  since $f$ is a monomorphism.  Hence the front and right squares are
  cartesian by descent.  But the front face is just our original
  square, and the fact that the right square is cartesian says that
  $g$ is a monomorphism.
\end{proof}

Finally, we recall also what is sometimes called the fundamental
theorem of topos theory \cite[Prop. 6.3.5.1]{LurieHT}.

\begin{prop} 
For any object $X$ in a topos $\cE$, the slice category $\cE_{/X}$ is a topos. 
\end{prop}

\section{Modalities}
\label{sec:modal-orth}

In this section, we introduce the prerequisite material on
factorization systems and modalities which will allow us to state our
generalized form of the Blakers-Massey theorem.  Homotopy-unique
factorization systems of the sort we consider here appear in a number
of places in the literature. For example, from a model category
theoretic perspective in \cite{bousfield1977constructions}, from a
higher categorical perspective in \cite{LurieHT} and \cite{JoyalNQC},
and from a type theoretic one in \cite[Chapter 7]{hottbook}.  We
recall some basic tools and ideas here in order to fix notation and
conventions.

\subsection{Factorization systems}
\label{sec:factorization}

\begin{defn}\label{pullbackorth}
  Let $f:A\to B$ and $g:X\to Y$ be two maps in a category
  $\cC$.   We say that $f$ and $g$ are \emph{orthogonal} 
  if the following square is cartesian in $\cS$:
  \begin{equation*} 
    \begin{tikzcd}
      \map{B}{X} \ar[d, "- \circ f"'] \ar[r, "g \circ -"] & \map{B}{Y} \ar[d, "- \circ f"] \\
      \map{A}{X} \ar[r, "g \circ -"']  & \map{A}{Y}
    \end{tikzcd}
  \end{equation*}
We denote this relation by $f \perp g$ and say that $f$ is {\it left orthogonal} to $g$ and that $g$ is {\it right orthogonal} to $f$.
\end{defn}

If $f \perp g$ then every commutative square
\begin{equation*}
\begin{tikzcd}
    A \ar[d,"f"'] \ar[r] & X \ar[d, "g"] \\
    B \ar[r] \ar[ru, dashed, "d"] & Y 
\end{tikzcd}
\end{equation*}
has a {\em unique} diagonal filler $d:B\to X$; indeed, the cartesian gap map 
\[ \map{B}{X}\to \map{A}{X} \times_{ \map{A}{Y}} \map{B}{Y} \] 
of the square in Definition~\ref{pullbackorth} is an isomorphism.
Of course, "uniqueness" means that the space of diagonal
fillers of the square is contractible.

If $\cM$ and $\cN$ are classes of maps in a category $\cC$, we will write $\cM\perp \cN$ if we have $u\perp f$ for every $u\in \cM$ and $f\in \cN$. Let us put
\begin{align*}
  \cM^\perp&:=\{f\in \cC \ |\ u\perp f \quad {\rm for \ every}\ u\in \cM\} \textrm{, and} \\
  {}^\perp\cN&:=\{u\in \cC \ |\ u\perp f \quad {\rm for \  every}\ f\in \cN \}.
\end{align*}
Then the relations $\cM\perp \cN$, $\cM \subset {}^\perp\cN$ and $\cN \subset \cM^\perp$
are equivalent.

\begin{defn}
  A {\em factorization system} on a category $\cC$ is the
  data of a pair $(\cL,\cR)$ of classes of maps in $\cC$ such that
  \begin{enumerate}
  \item every map $f$ in $\cC$ admits a factorization
    $f = \cR(f)\circ\cL(f)$ where $\cL(f)\in \cL$ and $\cR(f)\in \cR$,
    and
  \item $\cL^{\perp}=\cR$ and $\cL=\mbox{}^{\perp}\cR$.
  \end{enumerate}
Here, $\cL$ is called the {\it left class} and $\cR$ is called the {\it right class}.
\end{defn}

A well known example of a factorization system in
ordinary category theory is formed by the surjective and
injective functions in the $1$-category (in fact, $1$-topos) $\cS et$.
There is a similar factorization system in any topos $\cE$. 

\begin{defn}\label{def:coverage}
We say that a map in a topos $\cE$ is a \emph{cover} if it is left orthogonal to every monomorphism (defined in~\ref{def:mono}). 
We say that a family of maps $\{ f_i : X_i \to X \}_{i\in I}$ is a \emph{coverage} of the object $X$ if the resulting map $\bigsqcup_{i\in I} X_i\to X$ is a cover. 
\end{defn}

\begin{rem}\label{rem:coverage}
  Covers are referred to as \emph{effective epimorphisms} in \cite{LurieHT}.  We list
  some elementary facts about covers below.

\begin{enumerate}
 \item
Every map in a topos can be factored
as a cover followed by a monomorphism. 
 \item 
A map  $f : X \to Y$ in the category of spaces $\cS$ is a
cover if and only if the induced map $\pi_0(f):\pi_0 X\to \pi_0 Y$ is surjective.
Let $\term$ be the terminal object in $\cE$.
A pointed space $(X,x)$ is connected if and only if the map $x:\term \to X$
is a cover. 
  \item
A map $f : X \to Y$ in a topos $\cE$ is a cover if and only
if the base change functor $f^* : \cE_{/Y} \to \cE_{/X}$ is
conservative.  
   \item
A family of maps $\{ f_i : X_i \to X \}_{i\in I}$ is a coverage if and only if the family of functors 
$f^*_i:\cE_{/X} \to \cE_{/X_i}$ is collectively conservative, i.e. if the functor
  $$ (f^*_i)_{i\in I}:\cE_{/X} \to \prod_{i\in I}\cE_{/X_i} $$
is conservative.
\item 
Let $D:K\to \cE$ be a diagram in a topos. Then the family of canonical
maps $i_k:D(k)\to \colim(D)$ for $k\in K$ is a coverage.
\end{enumerate}
\end{rem}

Notice that, in light of the orthogonality requirement $\cL\perp \cR$, factorizations
are unique up to unique isomorphism: indeed, if $f=rl = r'l'$ are two $(\cL,\cR)$-factorizations of
a map $f:X\to Y$, then the following squares have a unique diagonal filler $d:Z\to Z'$ and $d':Z'\to Z$
respectively,  since $l\perp r'$ and $l'\perp r$.
\[
  \begin{tikzcd}
    X\ar[r,"l'"]\ar[d, "l"'] & Z'\ar[d,"r'"] \\
    Z \ar[r, "r"'] \ar[ur, dotted] & Y
  \end{tikzcd}
\quad \quad \quad
  \begin{tikzcd}
    X\ar[r,"l"]\ar[d, "l'"'] & Z\ar[d,"r"] \\
    Z' \ar[r, "r'"'] \ar[ur, dotted] & Y
  \end{tikzcd}
\]
We then have $d'd=1_Z$ and $dd'=1_{Z'}$, since the following squares have a unique diagonal filler.
\[
  \begin{tikzcd}
    X\ar[r,"l"]\ar[d, "l"'] & Z\ar[d,"r"] \\
    Z \ar[r, "r"'] \ar[ur, dotted] & Y
  \end{tikzcd} \quad \quad \quad
  \begin{tikzcd}
    X\ar[r,"l'"]\ar[d, "l'"'] & Z'\ar[d,"r'"] \\
    Z' \ar[r, "r'"'] \ar[ur, dotted] & Y
  \end{tikzcd}
\]
Thus, $d$ is an isomorphism and it is unique. 

Given a map $f : X \to Y$, we will occasionally write $\| f \|$ for
the object produced by factoring $f$ with respect to a factorization
system $(\cL, \cR)$ so that we have a commutative diagram
\[
  \begin{tikzcd}
    X \ar[r, "\cL(f)"] \ar[rr, bend right, "f"'] & \| f \| \ar[r, "\cR(f)"] & Y.
  \end{tikzcd}
\]
In such a situation, the intended factorization system will always
be clear from the context. The factorization $f=\cR(f)\cL(f)$
of a map in $\cC$ is functorial in $f\in \cC^\to$.
More precisely, from a commutative square $\alpha:f\to g$ 
\[
    \begin{tikzcd}
       A \ar[d, "f"'] \ar[r, "h"] \ar[dr, phantom, "\scriptstyle{\alpha}"] & C \ar[d, "g"] \\
       B \ar[r, "k"'] & D,
    \end{tikzcd}
  \]  
    we obtain two commutative squares
    \label{lem:LR-fact}
    \[
    \begin{tikzcd}
      A \ar[d, "\cL(f) "'] \ar[r, "h"] \ar[dr, phantom, "\scriptstyle{\cL(\alpha)}"] & C \ar[d, "\cL(g)"] \\
      \| f \| \ar[d, "\cR(f)"'] \ar[r, "\| \alpha \| "] \ar[dr, phantom, "\scriptstyle{\cR(\alpha)}"] &
      \| g \| \ar[d, "\cR(g)"] \\
       B \ar[r, "k"'] & D.
     \end{tikzcd}
   \]
This defines two functors $\cL,\cR:\cC^\to\to \cC^\to$. Let us denote by 
  $\cL^\to $ (resp. $\cR^\to $) the full subcategory of $\cC^\to$ whose objects are
  the maps in  $\cL$ (resp. in $\cR$). Then, $\cL:\cC^\to\to \cL^\to$
  and $\cR:\cC^\to\to \cR^\to$. 
  
  \begin{lem}\label{adjointnessLandR}
  The functor $\cL:\cC^\to\to \cL^\to$ is right adjoint to the inclusion $\cL^\to \subset \cC^\to$
  and the functor $\cR:\cC^\to\to \cR^\to$ is left adjoint 
  to the inclusion $\cR^\to \subset \cC^\to$.
  \end{lem}
  
  \begin{proof} The unit $\eta(f):f\to \cR(f)$ and the counit $\epsilon(f):\cL(f)\to f$ 
  of the adjunctions are the squares of the following diagram:
  \[
   \begin{tikzcd}
      A \ar[d, "\cL(f) "'] \ar[r, equal] \ar[dr, phantom, "\scriptstyle{\epsilon(f)}"] & A \ar[d, "f"] 
      \ar[dr, phantom, "\scriptstyle{\eta(f)}"] \ar[r,"\cL(f)"] & \| f \| \ar[d,"\cR(f)"] \\
     \| f \| \ar[r, "\cR(f)"'] & B \ar[r, equal] & B.
   \end{tikzcd}
   \]
  \end{proof}

The following elementary closure properties of the right and
left classes of a factorization system are standard.

\begin{lem}\label{lemma-stabilitybylimits}
Given a factorization system $(\cL,\cR)$, then 
\begin{enumerate}
  \item $\cL$ and $\cR$ contain all isomorphisms;
  \item $\cL$ and $\cR$ are closed under composition; 
  \item if $fg\in \cL$ and $g\in \cL$, then $f\in  \cL$;
dually, if $fg\in \cR$ and $f\in \cR$, then $g\in  \cR$;
  \item $\cL$ is stable by cobase change and $\cR$ is stable by base
change (when they are well defined);
  \item the subcategory  $\cL^\to $ of  $\cC^\to$ is closed under colimits
and the subcategory  $\cR^\to $ of  $\cC^\to$ is closed under limits.
\end{enumerate}
\end{lem}

Notice that the closure property (5) follows from Lemma \ref{adjointnessLandR}.

If $\cM$ is a class of maps in a category $\cC$ and $T\in \cC$
an object we let $\cM_T$ denote the class of maps in the slice
category $\cC_{/T}$ formed by all those maps whose image under the
forgetful functor from $\cC_{/T}$ to $\cC$ land in $\cM$.  Then
one can easily show that factorization systems are compatible with
slicing in the sense that

\begin{lem} \label{slicingfactsystem}
If $(\cL,\cR)$ is a factorization system in a category $\cC$, then the pair 
$({\cL}_T,{\cR}_T)$ is a factorization system in the category $\cC_{/T}$
for every object $T\in \cC$.
\end{lem}

A fundamental fact about presentable categories is the following.
\begin{prop}\label{pres-fact}
Let $S$ be a set of maps in a presentable category $\cC$.
Then the class $\cR=S^\bot$ is the right class of a factorization system with $\cL={}^\bot\cR$.
\end{prop}

\begin{proof}
  \cite[Proposition 5.5.5.7]{LurieHT}  
\end{proof}

\subsection{Operations on maps}
\label{sec:map-ops}

We fix a topos $\cE$ throughout.  Given two maps $u: A\to B$ and
$v: S\to T$ in $\cE$, we have a commutative diagram
\begin{equation*} 
  \begin{tikzcd}
    A\times S \ar[d, "u \times S"'] \ar[r, "A \times v"] & A\times T \ar[d, "u \times T"] \\
    B\times S \ar[r, "B \times v"'] & B\times T
  \end{tikzcd}
\end{equation*}
where eg. $A\times v$ stands for the map
  $$ \id_A\times v: A\times S\to A\times T. $$
We define the {\em pushout product of $u$ and $v$},
denoted $u \pp v$, as the cocartesian gap map of the square above.  Explicitly, then 

\[
u \pp v = (B\times S) \sqcup_{(A\times S)} (A\times T) \to B\times T.
\]

\begin{exmp}\label{example-po} 
The pushout product has a number of important special cases:
\begin{enumerate} 
\item \label{pppointed}
Recall that $\term$ is the terminal object of $\cE$.
  A \emph{pointed object} of $\cE$ is a pair $(A, a)$ where $a$ is a
  map $\term\to A$.  For $(A, a)$ and $(B, b)$ two pointed objects,
  one sees that the pushout product
  $a\pp b = (\term \to A)\pp (\term \to B)$ is the canonical inclusion
  of the wedge into the product:
  \[
    A \vee B \to A \times B.
  \]
\item \label{joinandpp}
  The {\em join} of two objects $A$ and $B$ in $\cE$,
  denoted $A\join B$, is the pushout of the diagram
  \[
    \begin{tikzcd}
      A\times B \ar[d] \ar[r] & B \\
      A 
    \end{tikzcd}
  \]
One finds immediately from the definition that:
  $$
  (A\to \term)\pp (B\to \term) = (A\join B)\to \term.
  $$
Let $S^0=\term \sqcup \term$ be the sphere of dimension 0 in $\cE$.
Then 
  $$ S^0\join A=\Sigma A$$  
is the unreduced suspension of $A$. In this way one can define spheres in any topos: 
set $S^{-1}=\ini$ and $S^{n+1}=S^0\join S^n$ for every $n\geq -1$.
If $s_n:S^n\to \term$, then $s_0\pp s_n=s_{n+1}$ for every $n\geq -1$.
 \item 
  For any map $u:A\to B$, 
  $$   s_0\pp u = \nabla u: B\sqcup_A B \to B $$
  is the \emph{codiagonal} of $u$, that is, the map defined by the following diagram
  with a pushout square,
  \[
    \begin{tikzcd}
      A \ar[d, "u"'] \ar[r, "u"] & B \ar[d, "i_2"]  \ar[ddr, equals, bend left] & \\
      B \ar[r, "i_1"'] \ar[rrd, equals, bend right] & B \sqcup_A B \ar[dr, "\nabla u"] & \\
      && B.
    \end{tikzcd}
  \]
  In particular, $\nabla(A\to \term)=s_0\pp (A\to \term)=(\Sigma A\to \term)$.
\item \label{fib-join}
 The pushout product $f\pp g$ of two maps $f:X\to A$
 and $g:Y\to B$ in a topos $\cE$ can be thought as the
  "external" join product of the fibers of $f$ and $g$.  Indeed, letting
  $a:\term\to A$ and $b:\term\to B$ be points of $A$ and $B$, we define
  $f^{-1}(a)$ and $g^{-1}(b)$ to be the fibers of $f:X\to A$ and $g:Y\to A$ at
  $a$ and $b$ respectively:
  \[
    \begin{tikzcd}
      f^{-1}(a) \ar[d] \ar[r] \pbmark & X \ar[d,"f"]\\
      \term \ar[r,"a"] & A
    \end{tikzcd}
    \qquad\textrm{and}\qquad
    \begin{tikzcd}
     g^{-1}(b) \ar[d] \ar[r] \pbmark & Y \ar[d,"g"]\\
      \term \ar[r,"b"] & B.
    \end{tikzcd}
  \]
   The fiber of $f\pp g$ at $(a,b):\term\to A\times B$ is similarly
  defined by the pullback
  \[
    \begin{tikzcd}
      (f\pp g)^{-1}(a,b) \ar[d] \ar[r] \pbmark & (A\times Y)\sqcup_{(X\times Y)} (X\times B) \ar[d]\\
      \term \ar[r,"{(a,b)}"] & A\times B.
    \end{tikzcd}
  \]
By the universality of colimits in $\cE$ we have
  \[
    (f\pp g)^{-1}(a,b) = f^{-1}(a) \join g^{-1}(b).
  \]
\item \label{pushout-product-and-base-change}
For an object $Z$ of $\cE$, we will denote by $\square_Z$ the pushout product in the ca\-te\-gory $\cE_{/Z}$.
If $f:W\to Z$ is a map in $\cE$, then the base change functor $f^*: \cE_{/Z}\to \cE_{/W}$
preserves pushout products: we have a canonical isomorphism
$$
f^*\!\left(\pprodovr{Z}{u}{v} \right)=\pprodovr{W}{\left(f^*u\right)}{\left(f^*v\right)}  
$$
for any two maps $u$ and $v$ in the topos $\cE_{/Z}$.
In fact, the map $u\ppz v$ in the topos $\cE_{/Z}$ is a base change of the map $u\pp v$, which can be seen as living 
over $Z\times Z$, along the diagonal $Z\to Z\times Z$.
\end{enumerate}
\end{exmp}

Dually, the \emph{pullback hom} $\ph{u}{f}$ of two maps $u:A\to B$ and
$f:X\to Y$ in $\cE$ is defined to be the cartesian gap map of the following
commutative square in $\cS$
\begin{equation*} 
   \begin{tikzcd}
    \map{B}{X} \ar[d, " \map{u}{X}"'] \ar[r, " \map{B}{f}"] & \map{B}{Y} \ar[d, " \map{u}{Y}"] \\
    \map{A}{X} \ar[r, " \map{A}{f}"']  & \map{A}{Y}.
  \end{tikzcd}
\end{equation*}

\begin{rem} \label{rem:pbh-detects-orth} 
Notice that the map $u$ is left orthogonal to the map $f$ if and only
if the map $\ph{u}{f}$ is invertible.
The codomain of $\ph{u}{f}$ is properly denoted by $\map{u}{f}$, as it is
exactly the space of maps $u\to f$ in the arrow category $\cE^\to$. 
\end{rem}

As defined, the pullback hom $\ph{u}{f}$
is a map of \emph{spaces}.  However, as a topos
$\cE$ is cartesian closed, it admits an internal hom
$\intmap{-}{-} : {\cE}^{\op} \times \cE \to \cE$.  We may thus define
an \emph{internal pullback hom}, denoted $\intph{u}{f}$ as the
cartesian gap map of the diagram
\begin{equation*} 
   \begin{tikzcd}
    \intmap{B}{X} \ar[d, " \intmap{u}{X}"'] \ar[r, " \intmap{B}{f}"] & \intmap{B}{Y} \ar[d, " \intmap{u}{Y}"] \\
    \intmap{A}{X} \ar[r, " \intmap{A}{f}"']  & \intmap{A}{Y}.
  \end{tikzcd}
\end{equation*}

The pushout product $- \pp -$ and the internal pullback hom
$\intph{-}{-}$ are part of a symmetric monoidal structure on the category
$\cE^\to$.  It follows that we have a natural isomorphism
  \[
    \map{u \pp v}{f} = \map{u}{\intph{v}{f}}
  \]
  for all $u, v, f \in \cE^\to$.  Furthermore, this isomorphism can itself
  be internalized, giving

\begin{prop}\label{appendixpp-adj2}
We have natural isomorphisms
  \[
    \intmap{u \pp v}{f} = \intmap{u}{\intph{v}{f}} \quad {\rm and} \quad
    \intph{u \pp v}{f} = \intph{u}{\intph{v}{f}}.
  \]
\end{prop}

A useful property of the pushout product is the following: if $u$ is
invertible, then so is the map $u\pp v$ for any map $v$. 
The pullback hom enjoys a similar absorption property: the map $\intph{v}{f}$ is invertible as soon as either $v$ or $f$ is.

\begin{exmp}\label{exam:diagonalsaspbhoms}
We note some useful special cases of the pullback hom.
\begin{enumerate} 
  \item 
For any map $f:X\to Y$, 
\begin{equation*}
  \intph{s_0}{f} = \Delta f:X\to X\times_YX
\end{equation*}
is {\it the diagonal of $f$}, that is, the map defined by the following diagram with a pullback square:
\[ \begin{tikzcd}
      X \ar[dr, "\Delta f"] \ar[ddr, equals, bend right] \ar[rrd, equals, bend left] & & \\
      & X \times_Y X \ar[d, "p_1"'] \ar[r, "p_2"] \pbmark & X \ar[d, "f"] \\
      & X \ar[r, "f"'] & Y
   \end{tikzcd} \]
\item For a pair of objects $A$ and $X$ in a cartesian closed
  category, the \emph{$A$-diagonal of $X$} is defined to be the map
  $$\Delta_A(X)=\intph{A\to\term}{X\to\term}:X=\intmap{1}{X}\to \intmap{A}{X}. $$
  Intuitively speaking, the map $\Delta_A(X)$ associates
  to each element of $X$ the map from $A$ to $X$ which is constant at
  that element.  Of course, to make this precise in full generality we
  must speak of generalized elements, but we will not dwell on that
  issue here.
\end{enumerate}
\end{exmp}

\begin{defn} \label{defstronglyorth}
Let $u:A\to B$ and $f:X\to Y$ be two maps in a topos $\cE$.
We say that $u$ is {\it internally left orthogonal} to $f$ (and $f$ is {\it internally right orthogonal} to $g$) if the following square 
\begin{equation*} 
  \begin{tikzcd}
    \intmap{B}{X} \ar[d, "\intmap{u}{X}"'] \ar[r, "\intmap{B}{f}"] & \intmap{B}{Y} \ar[d, "\intmap{u}{Y}"] \\
    \intmap{A}{X} \ar[r, "\intmap{A}{f}"']  & \intmap{A}{Y}
  \end{tikzcd}
 \end{equation*}
is cartesian.
In this case we will write $u \iperp f$.
\end{defn}

\begin{rem} \label{rem:intpbh-detects-int-orth} 
Analogous to Remark~\ref{rem:pbh-detects-orth} about the ordinary pullback hom and (external) orthogonality, internal orthogonality is detected by the internal pullback hom: $u$ is internally left orthogonal to $f$ if and only if the map $\intph{u}{f}$ is invertible.
\end{rem}

\begin{rem} \label{internalorth}
For every object $Z$ in $\cE$ one has 
% \map{Z,\intph{u}{f}}=\ph{Z\times u}{f}$.
$\map{Z}{\intph{u}{f}} =\ph{Z\times u}{f}$.
Hence it follows from the Yoneda lemma that the relation $u \iperp f $ is equivalent to the relation $Z\times u\perp f$ for every 
object $Z$ in $\cE$. In particular, $u\iperp f$ implies $u\perp f$.
\end{rem}

\begin{exmp} \label{extrongorth}
Here are two special cases of internal orthogonality.
\begin{enumerate} 
  \item 
The category of spaces $\cS$ is cartesian closed
and we have $\intph{u}{f}=\ph{u}{f}$ for any pair of maps $u,f\in \cS$.
It follows that the relations $u\iperp f$ and $u\perp f$
are the same in $\cS$.
  \item
Let $A$ and $X$ be objects in a cartesian closed category $\cC$.
By Remark~\ref{rem:intpbh-detects-int-orth}, $(A\to\term)\iperp (X\to\term)$ if and only if the $A$-diagonal of $X$ 
  $$ \Delta_A(X)=\intph{A\to\term}{X\to\term}:X\to \intmap{A}{X} $$ 
from Example~\ref{exam:diagonalsaspbhoms} is invertible.
\end{enumerate} 
\end{exmp}
 
\begin{prop} \label{pres-fact2} 
Let $S$ be a set of maps in a topos $\cE$.
Then $\cR:=S^{\iperp}$ is the right class of a factorization system $(\cL,\cR)$ with 
$\cL:={^\perp}\cR={}^{\iperp}\cR$.
\end{prop}

\begin{proof} Let $G$ be a set of generators of $\cE$ and put $G\times S:=\{Z\times u\ |\ Z\in G, u\in S\}$.
Let us show that $(G\times S)^\perp=S^{\iperp}$.

By Remark~\ref{internalorth}, the relation $u \iperp f $ is equivalent to the relation $Z\times u\perp f$ for every $Z$ in $\cE$. 
Thus $S^{\iperp}\subset (G\times S)^\perp$.

Conversely, if $f\in (G\times S)^\perp$ and $u\in S$, then we have $Z\times u \perp f$
for every $Z\in G$. This means that the map 
\[ \ph{Z\times u}{f}=\map{Z}{\intph{u}{f}} \]
is invertible for every object $Z\in G$.
It follows that the map $\intph{u}{f}$ is invertible, since $G$ is a set
of generators. Thus, $(G\times S)^\perp \subset S^{\iperp}$.

From Proposition~\ref{pres-fact} now follows that $\cR:=S^{\iperp}=(G\times S)^\perp$ is the right class of a factorization system $(\cL,\cR)$ 
with $\cL:={^\perp}\cR$. The relation $\cL\iperp \cR$ is left to the reader.
\end{proof}

\subsection{Connectedness and Truncation}
\label{sec:conn-trun}

The factorization system of covers and monomorphisms in a topos $\cE$
belongs to a whole family of factorization
systems corresponding to $n$-connected and $n$-truncated maps, to
which we now turn. 

\begin{defn} 
The notion of \emph{$n$-truncated map} $f : X \to Y$ in a topos $\cE$
is defined by induction on $n\geq -2$:
\begin{itemize}
\item $f$ is said to be $(-2)$-truncated if it is invertible.
\item $f$ is said to be $(n+1)$-truncated if the diagonal map
  \[ \Delta f : X \to X \times_Y X \]
  is $n$-truncated.
\end{itemize}
We write $T_n(\cE)$ for the class of $n$-truncated maps in a topos
$\cE$. An object $X\in \cE$ is said to be \emph{$n$-truncated} if the
map $X\to \term$ is $n$-truncated.
\end{defn}

By definition, a map is $(-1)$-truncated if its diagonal is an
isomorphism.  Hence a map is $(-1)$-truncated if and only if it is a
\emph{monomorphism} as defined in Definition~\ref{def:mono}. An
object $X\in \cE$ is $(-1)$-truncated if and only if the map
$X\to \term$ is a monomorphism.

A space $X$ is $n$-truncated if and only if $X$ is an $n$-th Postnikov
section, that is, if the homotopy group $\pi_k(X)$ vanish for
$k > n$ and all basepoints.  More generally, a map of spaces is
$n$-truncated if and only if all of its fibers are $n$-truncated
spaces.
Thus, a space is $(-2)$-truncated if it is contractible and a map
is $(-2)$-truncated if it is an equivalence.
A space is $(-1)$-truncated if it is either empty or contractible
and a map is $(-1)$-truncated if it is a monomorphism.
Finally, a space is $0$-truncated if it is equivalent to a discrete space
and a map is $0$-truncated if it is equivalent to
a covering space map.

\begin{rem} \label{truncatedoverB}
A map $f:X\to Y$ in a topos $\cE$ is $n$-truncated
if and only if the object $(X,f)$ of $\cE_{/Y}$ 
is $n$-truncated.
If $Z$ is an object of $\cE$, then a map $f:(X,p)\to (Y,q)$ 
in $\cE_{/Z}$ is $n$-truncated if and only
if the map $f:X\to Y$ in $\cE$ is $n$-truncated.
\end{rem}

Recall that in Example~\ref{example-po}(2), the $n$-sphere $S^n$ for
$n\geq -1$ is defined for an arbitrary topos $\cE$.  Moreover, if 
$s_n : S^n \to \term$ is the canonical map, then $s_{n+1}=s_0\pp s_n$
for every $n\geq -1$.

\begin{lem} \label{caractntruncated}
A map $f : X \to Y$ is $n$-truncated if and only if $s_{n+1} \iperp f$.
\end{lem}

\begin{proof}
 We wish to show that a map $f : X \to Y$ is $n$-truncated if and only if the map 
  $\intph{s_{n+1}}{f}$ is invertible.
  The proof proceeds by induction on $n \geq -2$.   The result is clear
  if $n=-2$, since $\intph{s_{-1}}{f}=f$. Let us suppose $n>-2$.
  By definition, $f$ is $n$-truncated if and only if the map $\Delta f$ is $(n-1)$-truncated.
  By the induction hypothesis,the latter holds if and only if 
  the map $\intph{s_n}{\Delta f} $ is invertible.
  But we have canonical isomorphisms:
  $$\intph{s_n}{\Delta f} =\intph{s_{n}}{\intph{s_0}{f}} =\intph{s_{n}\pp s_0}{f}= \intph{s_{n+1}}{f}.$$
Hence the map $\intph{s_n}{\Delta f}$ is invertible if and only if the map $\intph{s_{n+1}}{f}$  
  is invertible. This shows that $f$ is $n$-truncated if and only if
  the map $\intph{s_{n+1}}{f}$ is invertible.
  \end{proof}

\begin{defn}\label{defn:n-connectedmap}
A map $f : X \to Y$ in a topos $\cE$ is said to be \emph{$n$-connected} if it is left orthogonal to all $n$-truncated maps.
An object $X$ is said to be \emph{$n$-connected} if the map $X\to \term$ is $n$-connected.
We write $C_n(\cE)$ for the class of $n$-connected maps in a topos $\cE$.
\end{defn}

A map is $(-1)$-connected if and only if it is a cover.
Every map is $(-2)$-connected.

\begin{rem}
  \label{rem:conn-conv}
  Note that this definition of $n$-connectedness, while consistent with
  the standard usage for objects (that is, say for topological
  spaces), differs from the convention for maps: an
  $n$-connected map in our sense is $(n+1)$-connected map in the
  traditional sense, see for example~\cite[Page 302]{G92}).
\end{rem}

\begin{prop} \label{n-trunc-connectedfactsystem}
The pair $(C_n(\cE),T_n(\cE))$ is a factorization system in any topos 
$\cE$ and any $n\ge -2$. 
\end{prop}

\begin{proof} If $s_{n+1}=\{S^{n+1}\to\term\}$, then $T_n(\cE)={s_{n+1}}^{\iperp}$ by Lemma~\ref{caractntruncated}. The result then follows from Proposition~\ref{pres-fact2}. 
\end{proof}

In particular, $(C_{-2}(\cE),T_{-2}(\cE))$ is the factorization system of isomorphisms and all maps; $(C_{-1}(\cE),T_{-1}(\cE))$ is the factorization system of covers and mono\-mor\-phisms.

The following corollary shows how the operations of Section \ref{sec:map-ops}
interact with connectedness and truncation.

\begin{cor}\label{cor:raise-lower}
  Suppose that $u: A \to B$ is $m$-connected, $v : C \to D$ is $n$-connected
  and $f : X \to Y$ is $p$-truncated. Then:
  \begin{enumerate}
  \item $\intph{s_k}{f}$ is $(p-k-1)$-truncated. 
  \item $u \pp s_k$ is $(m+k+1)$-connected.
  \item $\intph{u}{f}$ is $(p-m-2)$-truncated.
  \item $u \pp v$ is $(m+n+2)$-connected.
  \end{enumerate}
\end{cor}

\begin{proof}
We use Proposition \ref{n-trunc-connectedfactsystem} and the adjunction formula in Proposition \ref{appendixpp-adj2}.
\begin{enumerate}  
\item
  By elementary properties of the join $s_k \pp s_\ell = s_{k+\ell+1}$ for all $k,\ell\ge -1$. Now, by Lemma~\ref{caractntruncated} $f$ is $p$-truncated if and only if
  \[s_{p+1}\iperp f \iff (s_{p-k} \pp s_k)\iperp f \iff
  s_{p-k} \iperp \intph{s_k}{f} \]
 Again by Lemma~\ref{caractntruncated} this is equivalent to the fact that $\intph{s_k}{f}$ is $(p-k-1)$-truncated.
  \item
Let $h$ be any $(m+k+1)$-truncated map.  Then
   \[ (u \pp s_k) \iperp h \iff u \iperp \intph{s_k}{h}. \]
  But $\intph{s_k}{h}$ is $(m+k+1)-k-1=m$-truncated by 1.
  \item
We have 
  $$s_{p-m-1} \iperp \intph{u}{f} \iff (s_{p-m-1} \pp u) \iperp f$$
  and $s_{p-m-1} \pp u$ is $(p-m-1)+m + 1 = p$-connected by 2.
  \item
Let $h$ be any $(m+n+2)$-truncated map, then
  \[
  (u \pp v) \iperp h \iff u \iperp \intph{v}{h}. 
 \]
  But the map $\intph{v}{h}$ is $(m + n + 2) - n - 2 = m$-truncated by 3.
\end{enumerate}
\end{proof}

\begin{prop}\label{prop:delta-conn}
For $n\geq -1$, a map in a topos $\cE$ is $n$-connected if and only if it is a cover and its diagonal is $(n-1)$-connected.
\end{prop}

\begin{proof}
  See \cite[Proposition 6.5.1.18]{LurieHT}.
\end{proof}

\subsection{Modalities}

We fix for this section a given topos $\cE$.

\begin{defn}
We say that a  factorization system $(\cL,\cR)$ in a topos $\cE$ is a {\em modality} if the class $\cL$ is stable under base change by any map in $\cE$.
\end{defn}

The right class of a factorization system is always closed under base
change by Lemma~\ref{lemma-stabilitybylimits}. Hence in a modality,
\emph{both} classes $\cL$ and $\cR$ are stable by base change.

\begin{exmp} The are many examples of modalities.
\begin{enumerate}
  \item
The factorization system of covers and monomorphisms in a topos $\cE$ is a modality. 
  \item
More generally, the factorization system $(C_n(\cE),T_n(\cE))$ of $n$-connected maps and $n$-truncated maps in a topos $\cE$ is a modality. It is a factorization system by
Proposition ~\ref{n-trunc-connectedfactsystem}. It only remains to check that $n$-connected maps are stable under base change by any map which we will leave to the reader.
  \item
If $(\cL,\cR)$ is a modality in a topos $\cE$, then for every object
$T \in \cE$ the induced factorization system $({\cL}_T,{\cR}_T)$ on
$\cE_{/T}$ described in Lemma~\ref{slicingfactsystem} is also a
modality.
\end{enumerate}
\end{exmp} 

\begin{exmp}
  Let $A \in \cS$ be a space.  A space $X$ is \emph{$A$-null} if the
  diagonal map
  $$ \Delta_A(X): X \to \map{A}{X}, $$
  defined in Example~\ref{exam:diagonalsaspbhoms}, is an equivalence,
  that is, if every function $A \to X$ is uniquely homotopic to a
  constant map.  Equivalently, a space $X$ is $A$-null if and only if
  $(A\to\term)\perp (X\to\term)$.  If $\cA$ is a set of spaces, a
  space $X$ is \emph{$\cA$-null} if it is $A$-null for every object
  $A\in \cA$.  If $S_{\cA}$ is the set of maps $A\to \term$ with
  $A\in \cA$, then $S_{\cA}^\perp=\cR_{\cA}$ is the right class of a
  factorization system $(\cL_{\cA},\cR_{\cA})$ by Proposition
  \ref{pres-fact}. Moreover, it can be shown that this factorization
  system is in fact a modality.  Indeed, the modality
  $(C_n(\cE),T_n(\cE))$ of the previous example is obtained from this
  construction by setting $A = S^{n+1}$.

  For a given space $X$, factoring the terminal map $X \to \term$
  produces an object $P_{\cA} X$ called the \emph{$\cA$-nullification}
  of $X$ which is initial among all $\cA$-null spaces admitting a map
  from $X$.  More generally, the stability of factorizations by
  pullback in a modality implies that the factorization of a map
  $f : X \to Y$ may be seen as a fiberwise application of the
  $\cA$-nullification functor.  Classical accounts of this
  construction may be found in \cite{may1980fibrewise},
  \cite{farjoun1995cellular} and \cite{chataur2006fiberwise}.

  The nullification functor $P_{\cA}$ is the source of the weak cellular
  inequalities of \cite{farjoun1995cellular} which are the main tool
  in the generalization of the Blakers-Massey theorem of
  \cite{ChachSchererWerndli}: a space $A$ \emph{kills}
  $X$, written $X > A$, if $P_A X = \term$.
\end{exmp}

The next class of examples of modalities is important for applications
to Goodwillie calculus in the companion paper~\cite{GBM2}.

\begin{exmp} \label{leftexactmodality}
  If $\cE$ is a topos and $L : \cE \to \cE$ is a left exact
  localization with unit $\eta : \id_{\cE} \to L$, then we obtain a
  modality $(\cL,\cR)$ on $\cE$ by taking $\cL$ to be the class of maps which are
  sent to isomorphisms by $L$.  The factorization of a map
  $f : X \to Y$ may be obtained by considering the pullback diagram
  \[
    \begin{tikzcd}
      X \ar[r, "\cL f"] \ar[dr, "f"'] & \| f \|  \ar[r, "\eta_X"] \ar[d, "\cR f"] \pbmark & LX \ar[d, "Lf"] \\
      & Y \ar[r, "\eta_Y"'] & LY
    \end{tikzcd}
  \]
The required properties of a modality follow easily from the hypothesis on $L$. 
\end{exmp}

\subsection{Dual Blakers-Massey theorem}
\label{sec:Dual-Blakers-Massey}

As a first application of the concept of modality, we derive the following
``dual'' Blakers-Massey theorem, which is in fact an elementary consequence
of the definition.

\begin{thm}[Dual Blakers-Massey]
  \label{prop:dbm}
  Let $(\cL, \cR)$ be a modality.  Suppose we are given a pullback square
  \[
    \begin{tikzcd}
      X \ar[r, "h"] \ar[d, "f"'] \pbmark & Z \ar[d, "g"] \\
      Y \ar[r, "k"'] & W
    \end{tikzcd}
  \]
and suppose that the map $k \pp g \in \cL$.  Then the cogap map $\cogap{k}{g} : Y \sqcup_X Z \to W$ is in $\cL$.
\end{thm}

\begin{proof}
  The pushout product $k \pp g$ is, by definition, the cogap map of the
  following square:
  \[
    \begin{tikzcd}
      Y \times Z \ar[d] \ar[r] & W \times Z \ar[d] \\
      Y \times W \ar[r] & W \times W
    \end{tikzcd}
  \]
  and one can easily check that by pulling back this square
  along the diagonal map $W \to W \times W$, we obtain our original
  square.  It follows, then, by universality of colimits that the
  pullback of the map $k \pp g$ is in fact the map $\cogap{k}{g}$.
  Since we have $k \pp g \in \cL$ by assumption, and $\cL$ is stable
  by base change in light of the fact that it is the left class of
  a modality, we are done.
\end{proof}

\begin{cor}[{\cite[Theorem 2.4]{G92}}]
  Suppose we are given a pullback square 
  \[
    \begin{tikzcd}
      X \ar[r, "h"] \ar[d, "f"'] \pbmark & Z \ar[d, "g"] \\
      Y \ar[r, "k"'] & W
    \end{tikzcd}
    \]
in $\cS$, where $g$ is $m$-connected and $k$ is $n$-connected, then the cogap map $\cogap{k}{g} : Y \sqcup_X Z \to W$ is $(m+n+2)$-connected.
\end{cor}

\begin{proof}
Just apply Proposition~\ref{prop:dbm} together with Corollary \ref{cor:raise-lower}(4).
\end{proof}

\subsection{Modalities and local classes}
\label{sec:modality-topos}

Recall from Definition~\ref{def:coverage}, that a family of maps  $\{g_i : Y_i \to Y \}_{i\in I}$
in a topos is called a \emph{coverage} if the induced map $\bigsqcup_{i\in I} Y_i\to Y$
is a cover.

\begin{defn}[{\cite[Prop. 6.2.3.14]{LurieHT}}] \label{prop:epi-local}
Let $\cM$ be a class of maps in a topos $\cE$ stable by base change. We will say that $\cM$ is {\it local} 
if for every coverage $\{g_i : Y_i \to Y \}_{i\in I}$ and every map $f:X\to Y$,
\begin{equation*}
    \begin{tikzcd}
      Y_i\times_Y X \ar[d, "g_i^*(f)"'] \ar[r] \pbmark &  \ar[d, "f"] X \\
      Y_i \ar[r, "g_i"'] & Y
    \end{tikzcd}
  \end{equation*}
if $g_i^*(f)\in \cM$ for all $i\in I$, then $f\in \cM$.
\end{defn}

\begin{rem}
  \label{rem:lcl-coverage}
  It is not hard to see that the above definition can be reformulated
  as follows: a class of maps $\cM$ is local if and only if
  it is closed under coproducts and 
   for any pullback square
  \begin{equation*} 
    \begin{tikzcd}
      X' \ar[d, "f'"'] \ar[r] \pbmark & X \ar[d, "f"] \\
      Y' \ar[r, "g"', two heads] & Y
    \end{tikzcd}
  \end{equation*}
with $g$ a cover, $f' \in \cM$ implies $f \in \cM$.
\end{rem}
  
\begin{rem}\label{localclassesinSpaces} 
If $\cM$ is a local class in the category
of spaces $\cS$, then a map $f:X\to Y$
belongs to $\cM$ if and only if
all its fibers $f^{-1}(y)$ (the maps $f^{-1}(y)\to 1$) belong to  $\cM$.
This is because the set of all maps $\term\to Y$
is a coverage of the space $Y$.
\end{rem}

\begin{rem}
Let $\cA$ be a class of objects in the category of spaces $\cS$ that is closed under isomorphisms (aka. weak homotopy equivalences). 
It is easy to verify that the class of maps having all their fibers in $\cA$ is a local class. 
Conversely, every local class $\cM$ in the topos $\cS$
is of this form for some class of objects $\cA$.
\end{rem}

\begin{prop}\label{lemma-localclass}
The classes $\cL$ and $\cR$ of a modality in a topos $\cE$ are local.
\end{prop}

\begin{proof}
Let us show that $\cR$ is local.
Given $f: X \to Y$ and a coverage $\{g_i : Y_i \to Y \}_{i\in I}$ such that $g_i^*(f)\in \cR$ for all $i\in I$, we need to show that $f\in \cR$.

For this, choose a factorization 
$f=pu:X\to Z\to Y$ with $p\in \cR$ and $u\in \cL$.
We will prove that $u$ is invertible, and hence
that $f\in \cR$. The base change functors
$g_i^*:\cE_{/Y}\to \cE_{/Y_i}$ are collectively
conservative, since the family $\{g_i\ | \ i\in I\}$
is a coverage. Hence it suffices to show
that the map $g_i^*(u)$ is invertible for every $i\in I$.
We have $g_i^*(u)\in \cL$ and $g_i^*(p)\in \cR$, 
since the classes $\cR$ and $\cL$ are closed under base changes.
Thus, $g_i^*(u)$ is invertible by uniqueness of a $(\cL,\cR)$-factorization,
since $g_i^*(f)=g_i^*(p)g_i^*(u)$ and $g_i^*(f)\in \cR$.
This proves that $u$ is invertible and hence that $f\in \cR$.
We have proved that $\cR$ is a local class.

A similar argument shows that $\cL$ is a local class.
\end{proof}

\begin{cor}\label{lemma-inSpaces} Let $(\cL,\cR)$
be a modality in the category of spaces $\cS$.
Then a map $f:X\to Y$ belongs to $\cL$ $($resp. $\cR)$
if and only if the map $f^{-1}(y)\to \term$
belongs to $\cL$ $($resp. $\cR)$ for every $y\in Y$.
\end{cor}

\begin{proof} This follows from Proposition~\ref{lemma-localclass} and Remark \ref{localclassesinSpaces}. \end{proof}

\subsection{Descent for $\cL$-cartesian squares}
\label{subsec:decentLcart}

A key tool in the proof of the generalized Blakers-Massey theorem are
descent properties of $\cL$-cartesian squares. The reader may wish
to compare the results of this section with those of \cite{CPS2005}
where similar notions are considered. 
 
\begin{defn}
  \label{def:l-cartesian}
  Let $\cL$ be a local class of maps in a topos $\cE$. We say that a
  commutative square 
  \[
    \begin{tikzcd}
      A' \ar[d, "u"'] \ar[r, "f'"] & B' \ar[d, "v"]  \\
      A \ar[r, "f"'] & B
    \end{tikzcd}
  \]
  is \emph{$\cL$-cartesian} if its cartesian gap map $(u,f'):A'\to A\times_B B'$
  belongs to $\cL$.
  \end{defn}

We will say that a morphism $\alpha:f\to g$
in $\cE^\to$ is {\it $\cL$-cartesian} if the corresponding
square in $\cE$ is $\cL$-cartesian.
By Lemma \ref{lem:l-pullbacks} below, the composite of two
$\cL$-cartesian morphisms is $\cL$-cartesian.

\medskip

\begin{prop}\label{prop:l-descent}
  Let $(\cL, \cR)$ be a modality on a topos $\cE$ and let 
  \[
    \begin{tikzcd}
      f \ar[r, "\alpha"] \ar[d, "\beta"'] & g \ar[d, "\gamma"] \\
      h \ar[r, "\delta"'] & k \pomark
    \end{tikzcd}
  \]
  be a pushout square in $\cE^\to$.  If the squares $\alpha$ and $\beta$ are
  $\cL$-cartesian, then so are the squares $\delta$ and $\gamma$.
\end{prop}
 
The proof of Proposition \ref{prop:l-descent} will be given after several
preparatory lemmas which establish some basic properties of $\cL$-cartesian
squares. The following lemma connects $\cL$-cartesian squares
to cartesian squares. Recall from Lemma \ref{adjointnessLandR} that the functor
$\cR:\cC^\to\to \cR^\to$ is left adjoint to the inclusion $\cR^\to \subset \cC^\to$.

\begin{lem} \label{lem:propofL-cart}  Let $(\cL,\cR)$ be a modality in a topos $\cE$. 
 If a square $\alpha:f\to g$ is $\cL$-cartesian, 
 \label{lem:l-equifib}
   \[
    \begin{tikzcd}
       A \ar[d, "f"'] \ar[r, "\alpha_1"] \ar[dr, phantom, "\scriptstyle{\alpha}"] & C \ar[d, "g"] \\
       B \ar[r, "\alpha_2"'] & D
    \end{tikzcd}
    \]
  then the square $\cR(\alpha):\cR(f)\to \cR(g)$ is cartesian.
 \end{lem}

 \begin{proof}
  Consider the diagram with two pullback squares 
  \[
    \begin{tikzcd}
      A \ar[d, "{(f,\alpha_1)}"'] \ar[dr,"\alpha_1"] &  \\
      B \times_D C \ar[r] \ar[d, "s"'] \ar[dr, phantom, "\scriptstyle{(a)}"] & C  \ar[d, "\cL(g)"] \\
      B \times_D \| g \| \ar[r] \ar[d, "t"'] \ar[dr, phantom, "\scriptstyle{(b)}"] & \| g \|  \ar[d,  "\cR(g)"] \\
      B \ar[r, "\alpha_2"'] & D.
   \end{tikzcd}
  \]
  By construction, $f=ts(f,\alpha_1)$.
  We have $s \in \cL$, since the square $(a)$ is cartesian, 
  and we have $t \in \cR$, since the square $(b)$ is cartesian. 
  Moreover, we have $(f,\alpha_1) \in \cL$, since the square $\alpha$
  is $\cL$-cartesian by hypothesis. Thus, $s(f,\alpha_1)\in \cL$
  and it follows that $s(f,\alpha_1)=\cL(f)$ and $t=\cR(f)$ 
  by uniqueness of the factorization $f=\cR(f)\cL(f)$.
  Thus, $\cR(\alpha)$ is the bottom square $(b)$. This shows that $\cR(\alpha)$
  is cartesian, since the square $(b)$ is cartesian.
   \end{proof}

\begin{lem}\label{lem:l-pullbacks}
Let $(\cL, \cR)$ be a modality on a topos $\cE$.
Then the composite of two $\cL$-cartesian squares is $\cL$-cartesian.
Moreover, if the composite of the squares in the following diagram
is $\cL$-cartesian, the left hand square $(a)$ is $\cL$-cartesian, and the map $f$ is a cover, then the right hand square $(b)$ is $\cL$-cartesian.
  \[
    \begin{tikzcd} \label{diagramof2square}
      A' \ar[d, "u"'] \ar[r, "f'"] \ar[dr, phantom, description, "\scriptstyle{(a)}"] 
      & B' \ar[d, "v"'] \ar[r, "g'"] \ar[dr, phantom, description, "\scriptstyle{(b)}"] & C' \ar[d, "w"] \\
      A \ar[r, two heads, "f"'] & B \ar[r, "g"'] & C
    \end{tikzcd}
  \]
\end{lem}

 \begin{proof} 
  Consider the following diagram with three pullback squares.
  \[
    \begin{tikzcd}[column sep=large, row sep=large]
     A' \ar[d, "{(u,f')}"'] \ar[dr, "f'"] & & \\
      A\times_B B' \ar[d, "{f^\join(v,g')}"'] \ar[r] & 
      B' \ar[d, "{(v,g')}"'] \ar[dr, "g'"] & \\
      A \times_C C' \ar[d] \ar[r]  &
      B \times_C C' \ar[d], \ar[r] & C' \ar[d, "w"] \\
      A \ar[r, two heads, "f"'] & B \ar[r, "g"'] & C
    \end{tikzcd}
  \]
  The cartesian  gap map of the square $(b)$ in the diagram \ref{diagramof2square} is $(v,g')$,
  the cartesian  gap map of the square $(a)$ is $(u,f')$
  and the cartesian  gap map of composite square $(a)+(b)$ is
  the composite $f^\join(v,g')(u,f')$.
  If the maps $(u,f')$ and $(v,g')$ belongs to $\cL$, then
  so is the map $f^\join(v,g')(u,f')$, since the class $\cL$
  is closed under base changes and composition.
  Conversely, if the maps $f^\join(v,g')(u,f')$ and
  $(u,f')$ belongs to $\cL$, then $f^\join(v,g')$
  belongs to $\cL$ by property (3) in Lemma~\ref{lemma-stabilitybylimits}. 
  Thus, $(v,g')\in \cL$ when $f$ is cover, since the
  class $\cL$ is local by Proposition~\ref{lemma-localclass}.
  \end{proof}

\begin{lem}
  \label{lem:lcl-descent}
 Let $\cM$ be a local class of maps in a topos $\cE$. Consider 
 a pushout diagram in $\cE^\to$ and suppose that the squares $\alpha$ and $\beta$ are cartesian.
  \begin{equation*} 
    \begin{tikzcd}
      f \ar[r, "\alpha"] \ar[d, "\beta"'] & g \ar[d, "\gamma"] \\
      h \ar[r, "\delta"'] & k \pomark
    \end{tikzcd}
  \end{equation*}
 Then $f, g, h \in \cM$ implies $k\in \cM $.
\end{lem}

\begin{proof} The commutative square above corresponds to a cube in $\cE$:
  \begin{equation}\label{thesquareisreallyacube}
    \begin{tikzcd}[bo column sep=large]
     A \ar[rr] \ar[dd, "f"'] \ar[dr] & & B \ar[dd, "g"', near start] \ar[dr, "\gamma_0"] & \\
      & C \ar[rr, crossing over, "\delta_0", near start] & & D \ar[dd, "k"'] \\
      E \ar[rr] \ar[dr] & & F \ar[dr, "\gamma_1"]  & \\
      & G \ar[rr, crossing over, "\delta_1"'] \ar[from=uu, crossing over, "h"', near start] & & H
    \end{tikzcd}
  \end{equation}
  By descent, the front and right faces of the cube are cartesian, since 
   the back and left faces are cartesian by hypothesis.
   Hence the base changes of $k$ along $\delta_1 : G \to H$ and $\gamma_1 : F \to H$
   belong to $\cM$. But the two maps $\delta_1$ and $\gamma_1$ form
  a coverage of $H$ by item (5) of Remark \ref{rem:coverage}. 
  Thus, $k\in \cM$, since $\cM $ is a local class.
 \end{proof}

\begin{lem}\label{lem:rel-patching}
Let $(\cL,\cR)$ be a modality in a topos $\cE$ and let
  \[
    \begin{tikzcd}
      f \ar[d, "\beta"'] \ar[r, "\alpha"] & g \ar[d, "\gamma"] \\
      h \ar[r, "\delta"'] & k
    \end{tikzcd}
   \]
be a pushout square in $\cE^{\to}$.  
Suppose that the square $\cR(\alpha)$ and $\cR(\beta)$ are cartesian.  
Then so are the squares $\cR(\delta)$ and $\cR(\gamma)$.
\end{lem}

 \begin{proof}
 As before, our square of arrows corresponds to a cube~\eqref{thesquareisreallyacube} in $\cE$
 such that the top and bottom squares are pushouts.
  Consider the following two-storey building, obtained by factoring the vertical
  maps $f$, $h$ and $g$ with respect to the modality $(\cL, \cR)$.
 \[
    \begin{tikzcd}[bo column sep=large]
      A \ar[rr] \ar[dd, "\cL(f)"'] \ar[dr] & & B \ar[dd, "\cL(h)"', near start] \ar[dr] & \\
      & C \ar[rr, crossing over] & & D \ar[dd, "s"]  \\
      \| f \| \ar[rr] \ar[dd, "\cR(f)"'] \ar[dr] & & \| h \| \ar[dd, "\cR(h)"', near start] \ar[dr] & \\
      & \| g \| \ar[rr, crossing over] \ar[from=uu, crossing over, "\cL(g)"', near start] & & 
      \| f \| \sqcup_{\|f\|} \| h \| \ar[dd, "t"] \\
      E \ar[rr] \ar[dr] & & F \ar[dr]  & \\
      & G \ar[rr, crossing over] \ar[from=uu, crossing over, "\cR(g)"', near start] & & H
    \end{tikzcd}
  \]
  The middle floor of the building is a pushout by construction.
  It follows that the upper and lower cubes are pushout diagram in $\cE^{\to}$,
  since the top and bottom floors of the building are pushout squares.
  The composite cube is also a pushout diagram in $\cE^{\to}$ for 
  the same reason.
  Thus, $k=ts$ since the cube ~\eqref{thesquareisreallyacube} is a pushout diagram in $\cE^{\to}$. 
  We have $s\in \cL$, since the class $\cL$ is closed under \emph{all} colimits in $\cE^{\to}$ by 
  Lemma \ref{lemma-stabilitybylimits}, and the upper cube exhibits
  $s$ as a colimit of maps in $\cL$.
  The back and left vertical faces
  of the lower cube are cartesian, since the squares $\cR(\alpha)$
  and $\cR(\beta)$ are cartesian by hypothesis. 
  Hence the front and right faces
  of the lower cube are cartesian by descent.
  Moreover, we have $t\in \cR$ by Lemma \ref{lem:lcl-descent}.
   Thus, $t=\cR(k)$ and $s=\cL(k)$, since $k=ts$.
  It follows that the front face of the lower cube is equal to $\cR(\delta)$
  and the right face is equal to $\cR(\gamma)$. This shows that the 
  squares $\cR(\delta)$ and $\cR(\gamma)$ are cartesian.
  \end{proof}

\begin{proof}[Proof of Proposition \emph{\ref{prop:l-descent}}]
  As in the previous lemmas, our square corresponds to a cube~\eqref{thesquareisreallyacube} in $\cE$.
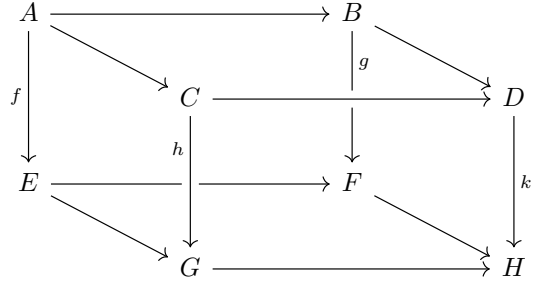
  \[
    \begin{tikzcd}[bo column sep=large]
      A \ar[rr] \ar[dd, "f"'] \ar[dr] & & B \ar[dd, "g", near start] \ar[dr] & \\
      & C \ar[rr, crossing over] & & D \ar[dd, "k"] \\
      E \ar[rr] \ar[dr] & & F \ar[dr]  & \\
      & G \ar[rr, crossing over] \ar[from=uu, crossing over, "h"', near start] & & H
    \end{tikzcd}
  \]
 By pulling back the bottom face of the cube along the map $k:D\to H$,
  we obtain the following decomposition
  
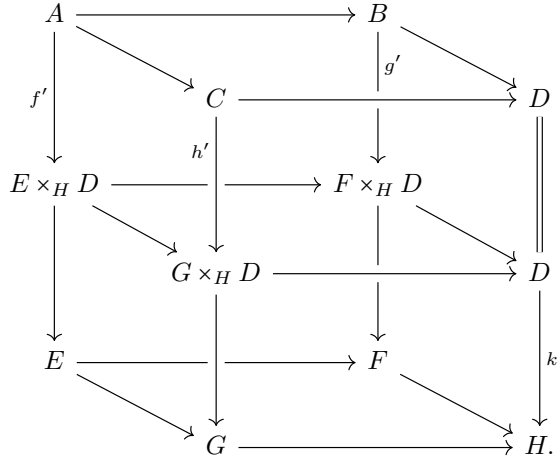
\[
    \begin{tikzcd}[bo column sep=large]
      A \ar[rr] \ar[dd, "f'"'] \ar[dr] & & B \ar[dd, "g'", near start] \ar[dr] & \\
      & C \ar[rr, crossing over] & & D \ar[dd, equal] \\
      E \times_H D \ar[rr] \ar[dd] \ar[dr] & & F \times_H D \ar[dr] \ar[dd] & \\
      & G \times_H D \ar[rr, crossing over] \ar[from=uu, crossing over, near start, "h'"'] & & D \ar[dd, "k"] \\
      E \ar[rr] \ar[dr] & & F \ar[dr] & \\
      & G \ar[rr] \ar[from=uu, crossing over] & & H.
    \end{tikzcd}
  \]
  The top cube of this diagram is a square in the category $\cE^\to$,
  \[
    \begin{tikzcd}
      f' \ar[d, "\beta'"'] \ar[r, "\alpha'"] & g' \ar[d, "\gamma'"] \\
     h' \ar[r, "\delta'"'] & 1_D
    \end{tikzcd}
  \]
  One finds easily from the composition of pullback squares that the
  gap map of the square $\beta'$ coincides with the gap map of
  the square $\beta$. Hence the square $\beta'$ is $\cL$-cartesian,
  since the square $\beta$ is $\cL$-cartesian by hypothesis. By
  Lemma~\ref{lem:l-equifib} it follows that $\cR(\beta')$ is
  cartesian.  Similarly, the face $\cR(\alpha')$ is cartesian.  So by
  Lemma~\ref{lem:rel-patching} the faces $\cR(\gamma')$ and
  $\cR(\delta')$ are cartesian.  We have $ \cR(1_D)=1_D$, since
  $1_D\in \cR$. Hence the maps $\cR(g')$ and $\cR(h')$ are
  invertible. It follows that $g'\in \cL$ and $h'\in \cL$.  Hence
  the square $\delta$ is $\cL$-cartesian, since $h'$ is the gap map of
  $\delta$.  Similarly, the square $\gamma$ is $\cL$-cartesian, since
  $g'$ is the gap map of $\gamma$.
 \end{proof} 

\begin{rem}\label{generalLdescent}
Proposition~\ref{prop:l-descent}
is a special case of a more general results
about $\cL$-cartesian squares. 
Let us  denote by $\Cart_{\cL}(\cE^\to)$ the (non-full)
subcategory of $\cE^\to$ whose morphisms
are $\cL$-cartesian.
Then the subcategory
$\Cart_{\cL}(\cE^\to)$ of $\cE^\to$ is closed under colimits.
As in the case of ordinary descent, the closure condition here means
two things: both that colimits exists in the subcategory
$\Cart_{\cL}(\cE^\to)$ and that they are preserved by the inclusion
functor $\Cart_{\cL}(\cE^\to)\to \cE^\to$.
\end{rem}

\section{The Blakers-Massey Theorem}
\label{sec:gen-bm}

In this section, we at last turn to the formulation and proof of our Generalized
Blakers-Massey Theorem, as well as some immediate applications.

\subsection{The statement}

\begin{thm}[Generalized Blakers-Massey]\label{thm:gen-bm}
Let $(\cL,\cR)$ be a modality in a topos $\cE$.
Consider a pushout square:
  \[
  \begin{tikzcd}
    Z \ar[d, "f"'] \ar[r, "g"] & Y \ar[d, "k"] \\
    X \ar[r, "h"'] & W \pomark
  \end{tikzcd}
  \]
  Suppose that $\Delta f \ppz \Delta g \in \cL$. Then the square is
  $\cL$-cartesian.
\end{thm}

\begin{rem}
  While the use of the relative pushout product $\ppz$ in the statement
  of the theorem is the most general result, in practice it is often simpler
  to check that $\Delta f \pp \Delta g \in \cL$, as will be done in all the
  applications of Section~\ref{sec:applications}.  Indeed, according to
  Example~\ref{example-po}\eqref{pushout-product-and-base-change}, the map $\Delta f \ppz \Delta g$ is 
  a base change of the map $\Delta f \pp \Delta g$ and hence contained
  in $\cL$ as soon as the latter is.
\end{rem}

Any set of maps in a topos generates a modality~\cite{Anel-Subramaniam:SOA}. 
This leads to the following reformulation of the main theorem.

\begin{cor}
The cartesian gap map of the square in Theorem~\emph{\ref{thm:gen-bm}}
belongs to the left class of the modality generated by the map $\Delta f\ppz \Delta g$.
\end{cor}

A special case of the generalized Blakers-Massey
theorem~\ref{thm:gen-bm} is obtained by considering the modality whose
left class consists of the isomorphisms and whose right class is given
by all maps.  We refer to this as the ``Little Blakers-Massey Theorem'',
and the statement is the following:
 
\begin{cor}[Little Blakers-Massey Theorem]\label{thm:littlebm}
Consider a pushout square in a topos $\cE$.
\begin{equation*} 
  \begin{tikzcd}
    Z \ar[d, "f"'] \ar[r, "g"] & Y \ar[d] \\
    X \ar[r] & W \pomark
  \end{tikzcd}
  \end{equation*}
If the map $\Delta f \ppz \Delta g$ is invertible, then the square is
cartesian. 
\end{cor}

Let us continue by describing the map
$\Delta f \ppz \Delta g$ in more detail.  
If $f:Z\to X$ is a map in $\cE$, then the map
$\Delta f:Z\to Z\times_X Z$ is a section of the projection
$p_1:Z\times_X Z\to Z$ onto the first factor.
Consequently, we can view 
$Z \times_X Z$ as a \emph{pointed} object of the slice topos
$\cE_{/Z}$ with structure map given by the first projection $p_1:Z\times_X Z\to Z$ and
basepoint given by the diagonal map $\Delta f : Z \to Z \times_X Z$.
Similarly, if $g:Z\to Y$ is another map in $\cE$, then 
we can view 
$Z \times_Y Z$ as a \emph{pointed} object of the slice topos
$\cE_{/Z}$ with structure map given by the first projection $p_1:Z\times_Y Z\to Z$ and
basepoint given by the diagonal map $\Delta g : Z \to Z \times_Y Z$:
\[
  \begin{tikzcd}
   Z\times_X Z \ar[d, "p_1"'] \\
    Z \ar[u, bend right, "\Delta f"']
  \end{tikzcd}
  \quad \quad 
  \begin{tikzcd}
   Z\times_Y Z \ar[d, "p_1"'] \\
    Z \ar[u, bend right, "\Delta g"']
  \end{tikzcd}
\]
Recall from Example~\ref{example-po}~(2) that
for any two pointed objects $a:\term\to A$ and $b:\term\to B$ of a topos $\cE$, the
pushout product $a \pp b : A\vee B\to A\times B$ is
the canonical inclusion of the wedge into the product.
This applies to the map $\Delta f \ppz \Delta g$ viewed in $\cE_{/Z}$:
  \begin{equation*}
  \begin{tikzcd}
  (Z\times_X Z) \sqcup_Z (Z\times_Y Z) \ar[dr,"p"'] \ar[rr, "\Delta f \ppz \Delta g"] && (Z\times_X Z) \times_Z (Z\times_Y Z) \ar[dl,"q"]  \\
   & Z &
  \end{tikzcd}
 \end{equation*}

It is instructive to compute the fibers of the maps $p$ and $q$ at $z\in Z$.
The fiber of the projection $p_1:Z\times_X Z\to Z$ at $z\in Z$ can be identified with the fiber $f^{-1}(f(z))$ of $f$ at $f(z)$ since the following square is cartesian:
  \begin{equation}\label{apullback} 
  \begin{tikzcd}
    Z\times_X Z \ar[d, "p_1"'] \ar[r, "p_2"] \pbmark & Z \ar[d, "f"] \\
    Z \ar[r, "f"'] & X 
  \end{tikzcd}
  \end{equation}
Similarly, the fiber of the projection $p_1:Z\times_Y Z\to Z$ at $z\in Z$ can
be identified with the fiber $g^{-1}(g(z))$ of $g$ at $g(z)$.
It follows that the fiber of the map $p$ at $z\in Z$
can be identified with the wedge $f^{-1}(f(z))\vee_z g^{-1}(g(z))$, while the fiber of $q$ can be identified with the product $ f^{-1}(f(z))\times g^{-1}(g(z))$. The map 
$\Delta f \ppz \Delta g$ is given fiberwise by the canonical maps 
\[ f^{-1}(f(z))\vee_z g^{-1}(g(z))\to f^{-1}(f(z))\times g^{-1}(g(z)) \]
as $z$ ranges over $Z$. To summarize, there is a commuting diagram
\[
  \begin{tikzcd}
    f^{-1}(f(z))\vee_z g^{-1}(g(z)) \ar[d] \ar[rr] &&  f^{-1}(f(z))\times g^{-1}(g(z))\ar[d] \\
    (Z\times_X Z) \sqcup_Z (Z\times_Y Z) \ar[rr, "\Delta f \ppz \Delta g"] \ar[d, "p"'] && (Z\times_X Z) \times_Z (Z\times_Y Z) \ar[d, "q"]  \\
     Z \ar[rr, equal] && Z
  \end{tikzcd}
  \]
whose vertical lines are fiber sequences over $z\in Z$.

For the upcoming proof we will need yet another description of $\Delta f \ppz \Delta g$.
By definition, it is the cogap map of the following square
\begin{equation}  \label{BMsquare}
  \begin{tikzcd}
    Z  \ar[d,"{\Delta f}"'] \ar[rrr,"{\Delta g}"] &&& Z \times_Y Z \ar[d,"{
    \Delta f \times_Z (Z \times_Y Z)}"] \\
    Z \times_X Z \ar[rrr,"{  (Z \times_X Z) \times_Z  \Delta g}"'] &&& (Z \times_X Z)\times_Z(Z \times_Y Z)
  \end{tikzcd}
\end{equation} 
where $(Z \times_X Z)\times_Z(Z \times_Y Z)$ is the fiber product over $Z$
of the projections 
  $$p_1:Z \times_X Z\to Z, \ (z_1,z_2)\mapsto z_1 $$ 
and 
  $$p_1:Z \times_Y Z\to Z, \ (z_1,z_2)\mapsto z_1 . $$

We will use a more explicit notation for the various maps in the square.
For example, $\Delta f=(1_Z,1_Z): Z\to Z \times_X Z $
since $(\Delta f)(z)=(z,z)$ for every $z\in Z$.
Similarly, $\Delta g=(1_Z,1_Z): Z\to Z \times_Y Z $
since $(\Delta g)(z)=(z,z)$ for every $z\in Z$.
The notation is not ambiguous as long as the domain
and the codomain of the maps are clear.
For example, the vertical map on the right in~\eqref{BMsquare} is 
  $$\Delta f \times_Z (Z \times_Y Z)=(p_1,p_1,p_1,p_2)$$
since 
  $$\bigl(\Delta f \times_Z (Z \times_Y Z)\bigr)(z_1,z_2)=(z_1,z_1,z_1,z_2) .$$
Beware that the fiber product over $Z$ is using the projection $p_1:Z \times_Y Z \to Z$ onto the first factor.
Similarly, 
  $$(Z \times_X Z) \times_Z  \Delta g=(p_1,p_2,p_1,p_1), $$
since 
 $$(Z \times_X Z) \times_Z  \Delta g(z_1,z_2)=(z_1,z_2,z_1,z_1) .$$
Beware again that the fiber product over $Z$ is using 
$p_1:Z \times_X Z \to Z$, and {\it not} $p_2$, because $p_1$ is the structure map of $Z \times_X Z$ over $Z$.  

With this notation, the square~\eqref{BMsquare} becomes 
\begin{equation}  \label{BMsquare2}
  \begin{tikzcd}
    Z  \ar[d,"{(1_Z,1_Z)}"'] \ar[rrr,"{(1_Z,1_Z)}"] &&& Z \times_Y Z 
    \ar[d,"{(p_1,p_1,p_1,p_2)}"] \\
    Z \times_X Z \ar[rrr,"{ (p_1,p_2,p_1,p_1)}"'] &&& (Z \times_X Z)\times_Z(Z \times_Y Z)
  \end{tikzcd}
\end{equation} 
In terms of the notation introduced in Section~\ref{sec:higher-cats},
the cogap map of the square has the following description
\begin{align*}
  \Delta f \ppz \Delta g &= \cogap{(p_1, p_2, p_1, p_1)}{(p_1,p_1,p_1,p_2)} \\
  & = \begin{bmatrix}
    p_1 & p_2 & p_1 & p_1       \\
    p_1 & p_1       & p_1 & p_2 
   \end{bmatrix}\\
                          &= (\cogap{p_1}{p_1}, \cogap{p_2}{p_1}, \cogap{p_1}{p_1}, \cogap{p_1}{p_2})
\end{align*}
which will prove useful for calculations.

\subsection{The proof}
\label{subsec:the-proof}

We will use the fact that the map $\Delta f \ppz \Delta g$ is isomorphic to two closely related maps: 
\begin{align*}
  \delta_{XY} &: (Z \times_X Z) \sqcup_Z (Z \times_Y Z) \to Z \times_X Z \times_Y Z \\
  \delta_{YX} &: (Z \times_X Z) \sqcup_Z (Z \times_Y Z) \to Z \times_Y Z \times_X Z 
\end{align*}
which are defined by
\begin{eqnarray*}
  \delta_{XY}  &:=& \cogap{(p_2,p_1,p_1)}{(p_2,p_2,p_1)} \\
              & = &\begin{bmatrix}
    p_2 & p_1 & p_1       \\
    p_2 & p_2       & p_1 
   \end{bmatrix}\\
             &=& (\cogap{p_2}{p_2}, \cogap{p_1}{p_2}, \cogap{p_1}{p_1}) \\
  \\
  \delta_{YX}  &:=& \cogap{(p_2,p_2,p_1)}{(p_2,p_1,p_1)} \\
  & = &\begin{bmatrix}
    p_2 & p_2 & p_1       \\
    p_2 & p_1       & p_1 
   \end{bmatrix}\\
              &=& (\cogap{p_2}{p_2}, \cogap{p_2}{p_1}, \cogap{p_1}{p_1})
\end{eqnarray*}

\begin{lem}
  \label{squareford}
  The maps $\Delta f \ppz \Delta g$, $\delta_{XY}$, and $\delta_{YX}$ are isomorphic as objects of the category $\cE^\to$ .
\end{lem}

\begin{proof} 
The maps $\delta_{XY}$ and $\delta_{YX}$
   are isomorphic since the following square commutes
  \[
    \begin{tikzcd}[column sep=large]
      (Z \times_X Z) \sqcup_Z (Z \times_Y Z) \ar[r, "\sigma_X \sqcup \sigma_Y"] \ar[d, "\delta_{XY}"'] & (Z \times_X Z) \sqcup_Z (Z \times_Y Z) \ar[d, "\delta_{YX}"] \\
      Z \times_X Z \times_Y Z \ar[r, "{(p_3,p_2,p_1)}"'] & Z \times_Y Z \times_X Z
    \end{tikzcd}
  \]
 and the maps $(p_3,p_2,p_1)$, $\sigma_X := (p_2, p_1)$ and
 $\sigma_Y := (p_2, p_1)$ are invertible.
 It is easy to see that the map
 $$\theta:=(p_2,p_1,p_2, p_3):
 Z\times_X Z\times_Y Z \to \left(Z\times_X Z\right)\times_Z \left(Z\times_Y Z\right)
 $$  
is invertible.
We have a commutative diagram 
  \[
    \begin{tikzcd}[column sep=large]
      (Z \times_X Z) \sqcup_Z (Z \times_Y Z) \ar[d, "\Delta f \ppz \Delta g"'] \ar[r, "\id \sqcup_Z \sigma_Y"] &
      (Z \times_X Z) \sqcup_Z (Z \times_Y Z) \ar[d, "\delta_{XY}"] \\
      (Z \times_X Z) \times_Z (Z \times_Y Z) & 
      \ar[l, "\theta"'] 
      Z \times_X Z \times_Y Z
    \end{tikzcd}
  \]
since
  \begin{eqnarray*}  \label{definitiondprime}
   \Delta f\ppz \Delta g     &=&  (\cogap{p_1}{p_1}, \cogap{p_2}{p_1}, \cogap{p_1}{p_1} ,\cogap{p_1}{p_2}) \\
                                              &=&\theta \circ (\cogap{p_2}{p_1}, \cogap{p_1}{p_1},\cogap{p_1}{p_2})\\
                                              &=&\theta \circ  (\cogap{p_2}{p_2 \circ \sigma_Y}, \cogap{p_1}{p_2 \circ \sigma_Y},\cogap{p_1}{p_1 \circ \sigma_Y}) \\
                                              &=& \theta \circ (\cogap{p_2}{p_2}, \cogap{p_1}{p_2},\cogap{p_1}{p_1}) \circ (\id \sqcup_Z \sigma_Y) \\
                                              &=&\theta \circ  \delta_{XY} \circ (\id \sqcup_Z \sigma_Y).
  \end{eqnarray*}
  This shows that the maps $\Delta f \ppz \Delta g$ and $\delta_{XY}$ are isomorphic, since the
   horizontal maps of the diagram are invertible.  
    \end{proof}

Our next step in proving Theorem~\ref{thm:gen-bm} will be to show that we may assume without loss of generality that the map $g$ is a cover.  To see this, choose a factorization of $g$
\[
\begin{tikzcd}
  Z \ar[rr, "g"] \ar[dr, two heads, "s"'] & & Y \\
  & Y' \ar[ur, rightarrowtail, "m"'] & 
\end{tikzcd}
\]
where $s$ is a cover and $m$ is a monomorphism. 
Now consider the diagram formed by taking pushouts:

\begin{equation}
  \label{eq:sm-diag}
  \begin{tikzcd}
    Z \ar[d, "f"'] \ar[r, two heads, "s"] \ar[dr, phantom, "\scriptstyle{(a)}"] &
    Y' \ar[d] \ar[r, rightarrowtail, "m"] \ar[dr, phantom, "\scriptstyle{(b)}"] &
    Y \ar[d] \\
    X \ar[r] & W' \pomark \ar[r] & W. \pomark
  \end{tikzcd}
\end{equation}

\begin{lem}
  \label{lem:surj-reduction}
  The gap map $\gap{f}{g} : Z \to X \times_W Y$ as an object of the category $\cE^{\to}$ is isomorphic 
  to the gap map $\gap{f}{s} : Z \to X \times_{W'}Y'$. Similarly, the
  map $\Delta g : Z \to Z \times_{Y} Z$ is isomorphic
  to the map $\Delta s : Z \to Z  \times_{Y'} Z$.
\end{lem}

\begin{proof}
  Since in each case, the maps have the same domain, it suffices to
  exhibit an isomorphism between the codomains (making the appropriate
  triangle commute, a detail we leave to the reader).  In the first case,
  consider the diagram:
  \[
  \begin{tikzcd}
    X \times_{W'} Y' \pbmark \ar[r] \ar[d] & Y' \ar[d] \ar[r, rightarrowtail, "m"] \pbmark & Y \ar[d] \\
    X \ar[r] & W' \ar[r] & W.
  \end{tikzcd}
  \]
  The left square is a pullback by construction.  The right
  square is a pullback by Proposition~\ref{prop:mono-pushout}.  It
  follows that the object $X \times_W Y$ is isomorphic to the object $X \times_{W'} Y'$.

  In the second case, we consider the diagram:
  \[
  \begin{tikzcd}
    Z \times_{Y'} Z \ar[r] \ar[d] & Z \ar[r, equal] \ar[d, two heads, "s"] & Z \ar[d, two heads, "s"] \\
    Z \ar[r, two heads, "s"'] \ar[d, equal] & Y' \ar[r, equal] \ar[d, equal] & Y' \ar[d, "m"] \\
    Z \ar[r, two heads, "s"'] & Y' \ar[r, "m"'] & Y.
  \end{tikzcd}
  \]
  The upper left square is cartesian by construction.  The bottom
  right square is cartesian since $m$ is a monomorphism.  The
  remaining two squares are trivially cartesian and hence so is the
  outer square.  We conclude that the object
  $Z \times_{Y} Z$ is isomorphic to the object $Z \times_{Y'} Z$ as claimed.
\end{proof}

An immediate consequence of the previous lemma is that we have
$(f,g) \in \cL$ if and only if $(f,s) \in \cL$. Similarly, 
we have $\Delta f \ppz \Delta g \in \cL$ if and only if
$\Delta f \ppz \Delta s \in \cL$.  Hence the Blakers-Massey theorem for the square $(a)$
of \eqref{eq:sm-diag} implies the Blakers-Massey theorem for the square
$(a)+(b)$.  It therefore suffices to prove Theorem \ref{thm:gen-bm} in
the case where $g$ (or, by symmetry, $f$) is a cover.

The proof of Theorem~\ref{thm:gen-bm} will hinge on a careful
analysis of the following cubical diagram
\begin{equation}
  \label{bm-cube}
  \begin{tikzcd}[bo column sep=large, row sep=large]
    (Z \times_X Z) \sqcup_Z (Z \times_Y Z) \ar[rr, "\rho_X"] \ar[dd, "\cogap{p_2}{p_2}"'] \ar[dr, "\rho_Y"']
    & & Z \times_X Z \ar[dd, near start, "gp_2"'] \ar[dr, "p_1"] & \\
    & Z \times_Y Z \ar[rr, crossing over, near end, "p_1"'] & & Z \ar[dd, "d"] \\
    Z \ar[rr, near end, "g"] \ar[dr, "f"'] & & Y \ar[dr, "k"]  & \\
    & X \ar[rr, crossing over, "h"'] \ar[from=uu, crossing over, near start, "fp_2"'] & & W
  \end{tikzcd}
\end{equation}
in which we set $d:=hf = kg$,
\begin{align*}
 \rho_X &:= \cogap{1_{Z\times_X Z}}{\Delta f \circ p_1}  \\
  &=\cogap{(p_1, p_2)}{(p_1,p_1)} \\
  & = \begin{bmatrix}
    p_1 & p_2       \\
    p_1 & p_1     
   \end{bmatrix} \\
 & =(\cogap{p_1}{p_1}, \cogap{p_2}{p_1})
\end{align*}
and 
\begin{align*}
\rho_Y &:= \cogap{\Delta g \circ p_1}{1_{Z \times_Y Z}}\\
   &=\cogap{(p_1, p_1)}{(p_1,p_2)} \\
  & = \begin{bmatrix}
    p_1 & p_1       \\
    p_1 & p_2     
   \end{bmatrix} \\
 & =(\cogap{p_1}{p_1}, \cogap{p_1}{p_2}).
\end{align*}
Let us pause to verify that the cube is indeed
commutative.  
First of all, the bottom face commutes by hypothesis.
The top face commutes since
\begin{eqnarray*}
  p_1 \circ \rho_X
                   &=& p_1 \circ (\cogap{p_1}{p_1}, \cogap{p_2}{p_1}) \\
                   &=& \cogap{p_1}{p_1} \\
                   &=& p_1 \circ (\cogap{p_1}{p_1}, \cogap{p_1}{p_2}) \\
                     &=& p_1 \circ \rho_Y .
\end{eqnarray*}
Next, the map $f: Z \to X$ coequalizes the
maps $p_1,p_2:Z\times_XZ\to Z$
since the square (\ref{apullback}) commutes.
Hence the front face commutes since
$dp_1=hfp_1=hfp_2$.  The right face commutes by a similar argument.
Finally, the left face commutes since
\[ fp_2\rho_Y = fp_2(\cogap{p_1}{p_1}, \cogap{p_1}{p_2}) = f \cogap{p_1}{p_2}= \cogap{fp_1}{fp_2} = \cogap{fp_2}{fp_2}= f\cogap{p_2}{p_2} \]
and similarly for the back face.

\begin{lem}\label{propqXqY} 
The top face of the cube~\emph{\eqref{bm-cube}}
  \begin{equation*}
    \begin{tikzcd}
      (Z \times_X Z) \sqcup_Z (Z \times_Y Z) \ar[d, "\rho_Y"'] \ar[r, "\rho_X"] & Z \times_X Z \ar[d, "p_1"] \\
      Z \times_Y Z \ar[r, "p_1"'] & Z \pomark
    \end{tikzcd}
  \end{equation*}
  is cocartesian.
\end{lem}

\begin{proof}
Let us first show that for any pair of pointed objects $(A,a)$
  and $(B,b)$ in a topos, the following square
  \[
    \begin{tikzcd}[column sep=large, row sep=large]
      A \vee B \ar[d, "\cogap{1_A}{0_B}"'] \ar[r, "\cogap{0_A}{1_B}"] & B \ar[d] \\
      A \ar[r] & \term \pomark
    \end{tikzcd}
  \]
  is cocartesian, where $0_A : A \to \term \xrightarrow{b} B$ and
  $0_B : B \to \term \xrightarrow{a} A$. 
  For this, consider the following commutative diagram
  \[
    \begin{tikzcd}[column sep=large, row sep=large]
  \term \ar[r, "a "] \ar[d, "b "']\ar[dr, phantom, "\scriptstyle{(a)}"] & A \ar[d, "A \vee b"] & \\
   B \ar[r, "a \vee B"] \ar[d] \ar[dr, phantom, "\scriptstyle{(b)}"] & \pomark  A \vee B \ar[d] 
   \ar[r, "\cogap{0_A}{1_B}"] \ar[dr, phantom, "\scriptstyle{(c)}"] & B \ar[d] \\
 \term \ar[r, "a "']  &      A \ar[r] & \term.
    \end{tikzcd}
  \]
  The square $(a)$ in this diagram is cocartesian by construction. The squares $(a)+(b)$
  and $(b)+(c)$ are trivialy cocartesian.
  Hence the square $(b)$ is cocartesian, since the squares $(a)$ and $(a)+(b)$ are cocartesian.
  It follows that the square $(c)$ is cocartesian, since $(b)$ and $(b+c)$ are cocartesian.
  Regarding $Z \times_X Z$ and $Z \times_Y Z$ as pointed objects of the topos
  $\cE_{/Z}$ with structure maps given by $p_1$ and base points given
  by the diagonal maps $\Delta f$ and $\Delta g$ respectively, we see
  that the top face of the cube~\emph{\eqref{bm-cube}} has the form above, since
  $\rho_X = \cogap{1_{Z\times_X Z}}{\Delta f \circ p_1}$ and
  $\rho_Y = \cogap{\Delta g \circ p_1}{1_{Z \times_Y Z}}$ by definition.
\end{proof} 

We are now nearly in a position to apply Proposition \ref{prop:l-descent}
on descent of $\cL$-cartesian squares to the cube~\eqref{bm-cube}.
For this we need to show that the 
back and left faces of the cube~\eqref{bm-cube} are $\cL$-cartesian.

\begin{lem} \label{acruciallemma} 
  Suppose $\Delta f \ppz \Delta g \in \cL$. Then the squares
  $(a)$ and $(b)$
  \[
  \begin{tikzcd}[row sep=large]
    Z \times_Y Z \ar[d, "fp_2"'] \ar[dr, phantom, "\scriptstyle{(a)}"] 
    & (Z \times_X Z) \sqcup_Z (Z \times_Y Z) \ar[l, "\rho_Y"'] \ar[r, "\rho_X"] \ar[d, "\cogap{p_2}{p_2}"'] 
    & Z \times_X Z \ar[d, "gp_2"] \ar[dl, phantom, "\scriptstyle{(b)}"] \\
    X & Z \ar[l, "f"] \ar[r, "g"'] & Y
  \end{tikzcd}
  \]
are $\cL$-cartesian.
\end{lem}

\begin{proof}
  The square $(a)$
  admits the following decomposition:
  \[
    \begin{tikzcd}
      (Z \times_X Z) \sqcup_Z (Z \times_Y Z) \ar[d, "\cogap{p_2}{p_2}"'] \ar[r, "\delta_{XY}"]
      \ar[dr, phantom, "\scriptstyle{(c)}"] &
      Z \times_X Z \times_Y Z \ar[r, "{(p_2, p_3)}"] \ar[d, "p_1"'] \ar[dr, phantom, "\scriptstyle{(d)}"] &
      Z \times_Y Z \ar[d, "fp_1"] \ar[r, "\sigma"] \ar[dr, phantom, "\scriptstyle{(e)}"] &
      Z \times_Y Z \ar[d, "fp_2"] \\
      Z \ar[r, equal] & Z \ar[r, "f"'] & X \ar[r, equal] & X 
    \end{tikzcd}
  \] 
  since
  \begin{eqnarray*}
    \sigma \circ (p_2,p_3) \circ \delta_{XY} &=& (p_3, p_2) \circ \delta_{XY} \\
                                             &=& (p_3, p_2) \circ (\cogap{p_2}{p_2}, \cogap{p_1}{p_2}, \cogap{p_1}{p_1}) \\
                                             &=& (\cogap{p_1}{p_1}, \cogap{p_1}{p_2}) \\
                                              &=& \rho_Y .
  \end{eqnarray*}
  On the other hand, the square $(d)$ in the diagram above is cartesian by construction and hence
  so is the square $(d)+(e)$ since $\sigma$ is an isomorphism.  Consequently,
  the map $\delta_{XY}$ is isomorphic to the gap map of the square $(c)+(d)+(e)=(a)$.  
  By Lemma~\ref{squareford}, $\delta_{XY}$ is isomorphic
  to the map $\Delta f \ppz \Delta g$ which is in $\cL$ by hypothesis.  Hence
  the square $(a)$ is $\cL$-cartesian.
  
  The square $(b)$ admits a similar decomposition which takes the form
  \[
    \begin{tikzcd}
      (Z \times_X Z) \sqcup_Z (Z \times_Y Z) \ar[d, "\cogap{p_2}{p_2}"'] \ar[r, "\delta_{YX}"] & Z \times_Y Z \times_X Z \ar[r, "{(p_2, p_3)}"] \ar[d, "p_1"']
      & Z \times_X Z \ar[d, "gp_1"] \ar[r, "\sigma"] & Z \times_X Z \ar[d, "gp_2"] \\
      Z \ar[r, equal] & Z \ar[r, "g"'] & Y \ar[r, equal] & Y 
    \end{tikzcd}
  \]
  A similar application of Lemma~\ref{squareford} to the map
  $\delta_{YX}$ completes the proof.
\end{proof} 

\begin{proof}[Proof of Theorem \emph{\ref{thm:gen-bm}}]
The bottom face of the cube cube~\eqref{bm-cube} is cocartesian by the hypothesis
and the top face is cocartesian by Lemma \ref{propqXqY}.
The previous lemma shows that the back and left faces of the cube are $\cL$-cartesian,
since $\Delta f \ppz \Delta g \in \cL$ by the hypothesis.
It then follows from Proposition~\ref{prop:l-descent} that the front and right faces are also
  $\cL$-cartesian.
  Now, the right face of the cube~\eqref{bm-cube} is the composite of
  the following two squares since $d=hf$.
  \[
    \begin{tikzcd}
      Z \times_X Z \ar[d, "p_1"'] \ar[r, "p_2"] & Z \ar[r, "g"] \ar[d, "f"'] & Y \ar[d, "k"] \\
      Z \ar[r, two heads, "f"'] & X \ar[r, "h"'] & W
    \end{tikzcd}
  \]
  Recall from Lemma \ref{lem:surj-reduction} that we may assume that the map $f: Z \to X$ is a cover. 
  The left hand square of the above diagram is cartesian by
  construction and the composite square is $\cL$-cartesian by the previous
  considerations. It follows from Lemma~\ref{lem:l-pullbacks} that the
  right hand square is $\cL$-cartesian, since $f:Z\to X$ is a cover.
  \end{proof} 

\subsection{Applications}
\label{sec:applications}

Our main application is a Blakers-Massey theorem in the context of Goodwillie's calculus of homotopy functors. This is developed in our second joint article~\cite{GBM2}. It was indeed this application to Goodwillie calculus that motivated the whole project outlined in the two papers.

Let us now show how easily the classical Blakers-Massey theorem follows from our main theorem.
The classical theorem in the category of spaces $\cS$
is a special case of the following Blakers-Massey theorem in an arbitrary topos. Recall from Remark~\ref{rem:conn-conv} that our topos theoretic definition of $n$-connected map differs from the classical convention in homotopy theory by a shift of one.

\begin{cor}[Classical Blakers-Massey for topoi]
Given a pushout square 
\[
\begin{tikzcd}
  A \ar[r, "g"] \ar[d, "f"'] & C \ar[d] \\
  B \ar[r] & D \pomark
\end{tikzcd}
\]
in a topos, such that $f$ is $m$-connected and $g$ is $n$-connected, then the cartesian gap
map $\gap{f}{g}:A \to B \times_D C$ is $(m+n)$-connected.
\end{cor}

\begin{proof}
Example~\ref{exam:diagonalsaspbhoms}(1) identifies the diagonal map $\Delta f$ with $\ph{s_0}{f}$ which is $(m-1)$-connected by Proposition~\ref{prop:delta-conn}.  Similarly, $\Delta g$ is $(n-1)$-connected. Since the $(m+n)$-connected maps form a modality, and $\Delta f \pp \Delta g$ is $(m+n)$-connected by Corollary \ref{cor:raise-lower} (4), the result now follows from Theorem \ref{thm:gen-bm}.
\end{proof}

As further application we explain how our generalized version yields the improvement of the Blakers-Massey theorem by Chach{\'o}lski, Scherer and Werndli in \cite{ChachSchererWerndli}. We will need some preparation.

In the category $\cS$, the modality generated by a map $u:A\to B$ coincides with the modality generated by the maps $u^{-1}(b)\to \term$ for all $b\in B$.
This follows from Remark \ref{localclassesinSpaces}.

If $x$ and $y$ are two points of a space $X$, let us denote by $X(x,y)$ the space of paths $x\to y$.
By construction, $X(x,y)$ is the fiber of the diagonal map $X\to X\times X$ at $(x,y)\in X\times X$.
Notice that $X(x,y)= X(x,x)= \Omega_x X$ when $x$ and $y$ are homotopic and that $X(x,y)=\emptyset$ otherwise.
If $(A,a)$ is a pointed space, then $A(x,a)$ is the fiber at $x\in A$
of the map $a:\term \to A$.

\begin{lem}\label{fiberofcanonical} Let $(A,a)$ and $(B,b)$ be pointed spaces.
Then the fiber of the canonical map $A\vee B\to A\times B$ at $(x,y)\in A\times B$ is the join $A(x,a)\join B(y,b)$.
\end{lem}

\begin{proof} One has $A\vee B\to A\times B=(a:\term \to A)\pp (b:\term \to B)$.
The fiber $(a\pp b)^{-1}(x,y)$ is the join of the fibers $a^{-1}(x)$ and $b^{-1}(y)$ by Example \ref{example-po}(4).
The result follows, since $a^{-1}(x)=A(x,a)$ and $b^{-1}(y)=B(y,b)$.
\end{proof}

Finally, we now rederive the weak cellular inequality of Chach{\'o}lski, Scherer and Werndli \cite{ChachSchererWerndli} in the case of a pushout.  For a pointed space $(A,a)$ we denote the \emph{set} of spaces $A(x,a)$ for $x\in A$ by $\cP(A,a)$.

\begin{thm} \label{BMlemmaforspaces} In the category of spaces, the cartesian gap map of a pushout
 \[
    \begin{tikzcd}
      Z \ar[r, "g"] \ar[d, "f"'] & X \ar[d] \\
      Y \ar[r] & W \pomark
    \end{tikzcd}
  \]
belongs to the modality generated by the set of maps $S\join T\to \term$ for $S\in \cP(f^{-1}f(z),z)$, $T\in \cP(g^{-1}g(z),z)$ and $z\in Z$.
\end{thm}

\begin{proof} We saw above that $\Delta f \ppz \Delta g$, viewed fiberwise over $Z$, can be regarded as a sum
over $z\in Z$ of the canonical map
\[i_z: f^{-1}(f(z))\vee_z g^{-1}(g(z))\to f^{-1}(f(z))\times g^{-1}(g(z)) \]
Hence the modality generated by $\Delta f \ppz \Delta g$
is also generated by the maps $i_z$ for all $z\in Z$.
By Lemma \ref{fiberofcanonical}, the fibers of $i_z$
are of the form $S\join T$, where $S$ is a fiber of the
map $z:\term \to f^{-1}(f(z))$ and $T$ is a fiber of the
map $z:\term \to g^{-1}(g(z))$.
Hence the modality generated by the map $\Delta f \ppz \Delta g$
is generated by the maps $S\join T\to \term$ for $S\in \cP(f^{-1}f(z),z)$,
$T\in \cP(g^{-1}g(z),z)$ and $z\in Z$.
\end{proof}

\end{document}